\documentclass{article}
\usepackage[preprint]{neurips_2025}

\usepackage[utf8]{inputenc} 
\usepackage[T1]{fontenc}    
\usepackage{hyperref}       
\usepackage{url}            
\usepackage{booktabs}       
\usepackage{amsfonts}       
\usepackage{nicefrac}       
\usepackage{microtype}      
\usepackage{xcolor}         
\usepackage{amsmath}
\usepackage{amsthm, amssymb}
\usepackage{subfig}
\usepackage{graphicx}
\usepackage{hyperref,amsmath,amssymb,subfig,graphicx,booktabs,algorithm,algorithmic,natbib}

\usepackage{amsmath,amsfonts,bm,amsthm}
\newcommand{\W}[2]{\mathcal W_2^2(#1,#2)}
\newcommand{\xp}{\tilde X_{t,i}}
\newcommand{\xpi}[1][i]{ {\tilde X^+}_{t,#1}}
\newcommand{\z}[1][i]{Z_{t,#1}}
\newcommand{\zt}[1][i]{\tilde Z_{t,#1}}
\newcommand{\g}{\mathcal G}
\newcommand{\e}{\mathbb E}
\newcommand{\gradf}[1]{\nabla f(#1)}
\newcommand{\bern}{H_{t,i}-1/k}
\newcommand{\mtk}{M_{t,k}}
\newcommand{\xt}{\tilde X_{t,0}}
\newcommand{\y}{Y_{t,i}}
\newcommand{\is}{\mathbf 1_\mathcal E}

\newcommand{\ib}{\mathbf 1_{5\beta^2>1}}
\newcommand{\p}{\mathbb P}

\newcommand{\ebk}{\frac{\eta}{k}}
\newcommand{\txtki}{\tilde X_{t,i}}
\newcommand{\x}{X_{t,i}}
\newcommand{\tutki}{\tilde U_{t,i}}
\newcommand{\tvtki}{\tilde V_{t,i}}
\newcommand{\twtki}{\tilde W_{t,i}}
\newcommand{\utkt}{\tilde U_{t,0}}
\newcommand{\vtkt}{\tilde V_{t,0}}
\def\am{A_\mathcal M}
\def\gm{G_\mathcal M}
\def\Gm{\Gamma_\mathcal M}
\def\m{\mathcal M}

\theoremstyle{definition}
\newtheorem{assumption}{Assumption}
\newtheorem{remark}{Remark}
\newtheorem{theorem}{Theorem}
\newtheorem{corollary}{Corollary}
\newtheorem{lemma}{Lemma}
\newtheorem{proposition}{Proposition}

\def\vI{\mathbf I}
\def\law{\mathrm{Law}}
\def\E{\mathbb E}

\def\eqref#1{equation~\ref{#1}}

\def\1{\bm{1}}

\DeclareMathAlphabet{\mathsfit}{\encodingdefault}{\sfdefault}{m}{sl}
\SetMathAlphabet{\mathsfit}{bold}{\encodingdefault}{\sfdefault}{bx}{n}

\def\sI{{\mathbb{I}}}

\def\sN{{\mathbb{N}}}

\def\sR{{\mathbb{R}}}

\newcommand{\KL}{D_{\mathrm{KL}}}

\title{Poisson Midpoint Method for Log Concave Sampling: Beyond the Strong Error Lower Bounds}

\author{
    Rishikesh Srinivasan \\
    Google DeepMind \\
    rishikeshsrini@google.com 
    \And
    Dheeraj Nagaraj \\
    Google DeepMind \\
    dheerajnagaraj@google.com 
  }
\begin{document}
\maketitle

\begin{abstract}
We study the problem of sampling from strongly log-concave distributions over $\mathbb{R}^d$ using the Poisson midpoint discretization (a variant of the randomized midpoint method) for overdamped/underdamped Langevin dynamics. We prove its convergence in the 2-Wasserstein distance ($\mathcal W_2$), achieving a cubic speedup in dependence on the target accuracy ($\epsilon$) over the Euler-Maruyama discretization, surpassing existing bounds for randomized midpoint methods. Notably, in the case of underdamped Langevin dynamics, we demonstrate the complexity of $\mathcal W_2$ convergence is much smaller than the complexity lower bounds for convergence in $L^2$ strong error established in the literature.
\end{abstract}

\section{Introduction}
Sampling from a density $\pi(x) \propto \exp(-f(x))$ over $\sR^d$ is of fundamental interest in physics, economics, and finance \citep{johannes2010mcmc, von2011bayesian,kobyzev2020normalizing}. Applications in computer science include volume computation \citep{vempala2010recent} and bandit optimization \citep{russo2018tutorial}. 

A popular approach is Langevin Monte Carlo (LMC) which is the Euler-Maruyama discretization of the continuous time Itô Stochastic Differential Equation (SDE) called (overdamped/underdamped) Langevin Dynamics. The convergence of LMC has been extensively studied in the literature \citep{durmus2018analysislangevinmontecarlo,vempala2022rapidconvergenceunadjustedlangevin,erdogdu2021convergencelangevinmontecarlo,cheng2017convergencelangevinmcmckldivergence,cheng2018-underdamped-lmc,dalalyan2020sampling,altschuler2025shiftedcompositionivunderdamped} 
under various assumptions on the target density $\pi,$ such as log-concavity and isoperimetry. The randomized midpoint discretization for Langevin dynamics (RLMC), introduced by Shen \& Lee \cite{shen2019randomizedmidpointmethodlogconcave} and developed further by \cite{yu2023langevinmontecarlostrongly,he2021ergodicitybiasasymptoticnormality,altschuler2024shiftedcompositioniiilocal,altschuler2025shiftedcompositionivunderdamped} considers a more sophisticated alternative to LMC. This is a randomized discretization which reduces the bias in the estimation of the Ito integral while introducing variance, leading to faster convergence bounds than for LMC. The Poisson Midpoint Method for Langevin dynamics (PLMC) was introduced by Kandasamy \& Nagaraj \cite{kandasamy2024poissonmidpointmethodlangevin} as a variant of RLMC. While \cite{kandasamy2024poissonmidpointmethodlangevin} considered the convergence of PLMC under general conditions (beyond strong log-concavity and isoperimetry) for the total variation distance via entropic central limit theorem style arguments.

The literature has focused on understanding the sharp limits to the computational complexity of sampling for various classes of algorithms, in terms of various problem parameters. In the case of strongly log-concave sampling, the work of Cao et al. \cite{Cao_2021} established lower bounds for the strong $L^2$ error of randomized algorithms which discretize Underdamped Langevin Dynamics (ULD).  Strong $L^2$ error is the $L^2$ distance between the continuous time Itô SDE solution at time $T$ and the sampling algorithm output whenever they are driven by \textit{same} Brownian motion. This demonstrated that RLMC is an optimal discretization of ULD with respect to dimension and accuracy (up to log factors), in terms of the strong $L^2$ error. However, sampling algorithm guarantees generally consider `weak' notions of distance such as total variation distance, Wasserstein distance, or the KL divergence between the law of algorithm output and the target. In particular, Wasserstein-2 distance bounds can consider the $L^2$ distance between algorithm output and the continuous time SDE driven by \textit{different but arbitrarily coupled} Brownian motions.

In this work, we revisit the complexity of PLMC for strongly log-concave sampling in order to obtain better insights into the fundamental computational limits of sampling algorithms. We provide a sharp analysis via coupling arguments to obtain better convergence guarantees, which involves a tight bound on the $\mathcal W_2$ distance between a Gaussian random-variable and a perturbed Gaussian random-variable. This is adopted from Alex Zhai's proof of the Central Limit Theorem in $\mathcal W_2$ distance \citep{zhai2017highdimensionalcltmathcalw2distance}, and leads to a substantial improvement in convergence guarantees.

\subsection{Our contributions}
We consider the computational complexity of sampling from a log-concave target distribution $\pi(x) \propto \exp(-f(x))$ over $\mathbb{R}^d$, with $f$ well-conditioned (Assumption \ref{well_conditioned_F}) with condition number $\kappa$ and strong convexity constant $\alpha$. Many classes of algorithms have been proposed and studied to this end. We study PLMC, which is a randomized algorithm for the discretization of Langevin Dynamics, with access only to $\nabla f(x)$ for arbitrary $x \in \mathbb{R}^d$. The computational complexity is measured in terms of number of evaluations of $\nabla f(x)$ (the oracle complexity).

\textbf{Limits of Sampling:} Recent works have aimed to understand the best possible computational complexity of sampling such that $\mathcal W_2^2(\mathsf{output}, \pi) \leq \frac{\epsilon^2 d}{\alpha}$ in terms of $\epsilon, d $ and $\alpha$. Cao et al. \cite{Cao_2021} show that randomized algorithms which discretize ULD require an oracle complexity of $\Omega(\epsilon^{-2/3})$ to converge in strong $L^2$ error; and RLMC achieves this rate up to logarithmic factors. It was thus widely believed in the literature that the rate of $\tilde {\mathcal O}(\epsilon^{-2/3}),$ achieved by RLMC, might also be the optimal convergence rate in $\mathcal W_2.$ The main contribution of our work is that we show it is possible to obtain $\tilde{\mathcal O}(\epsilon^{-1/3})$ complexity. Specifically, we show that:

\textbf{1.} Overdamped PLMC has an oracle complexity of $\tilde{\mathcal O}\bigr[\tfrac{\kappa^{4/3}+\kappa d^{1/3}}{\epsilon^{2/3}}\bigr]$ (Corollary \ref{OLMCcomplexity}).

\textbf{2.} Underdamped PLMC has an oracle complexity of $\tilde {\mathcal O}\big[\frac{\kappa^{7/6}d^{1/6}}{\epsilon^{1/3}}+\frac{\kappa^{\frac{11p+6}{8p+6}}d^{\frac{p}{4p+3}}}{\epsilon^{\frac{p+2}{4p+3}}}\big]$ (Corollary \ref{ULMCcomplexity}). Here $p \in \mathbb N$ is arbitrary. For $p \geq 3$, this gives a complexity of $\tilde{\mathcal O}(\epsilon^{-1/3})$. 

The best known convergence rate for overdamped LMC (in $\mathcal W_2$) is an oracle complexity of $\tilde {\mathcal O}(\epsilon^{-2})$ \cite{durmus2018analysislangevinmontecarlo}. The convergence guarantee of $\tilde {\mathcal O}(\epsilon^{-2/3})$ for overdamped PLMC is thus a cubic improvement in $\epsilon$ dependence. The best known convergence rate for underdamped LMC (in $\mathcal W_2$) is an oracle complexity of $\tilde {\mathcal O}(\epsilon^{-1}).$ The convergence rate of $\tilde {\mathcal O}(\epsilon^{-1/3})$ for underdamped LMC is again a cubic improvement.  A detailed comparison of results is in Tables \ref{tab:comparison-overdamped} and \ref{tab:comparison-underdamped}.

Concurrent work \citep[Theorem 5.11]{altschuler2025shiftedcompositionivunderdamped} claims an oracle complexity of $\tilde {\mathcal O}(\kappa^{5/6}d^{5/3}/\epsilon^{2/3})$ to achieve $\text{KL}(\mathsf{output}||\pi) \leq \epsilon^2$ for RLMC. This implies a complexity of $\tilde {\mathcal O}(\kappa^{5/6}d^{4/3}/\epsilon^{2/3})$ to achieve $\mathcal W_2^2 \leq \frac{\epsilon^2 d}{\alpha} $ via the $T_2$ inequality. This improves the dependence on $\kappa$ from $\kappa^{7/6}$ to $\kappa^{5/6}$ as compared to prior works, but with a worse dependence on $d$ and the same complexity in $\epsilon$.

\textbf{Comparison to Strong Error Lower Bounds:}
The work of Cao et al.~\cite{Cao_2021} proves a lower bound for the discretization error of underdamped Langevin dynamics via randomized algorithms. In particular, given a probability space $\Omega$, $f$ satisfying Assumption \ref{well_conditioned_F} and a Brownian Motion $B_t(\omega): \Omega \to \sR^d$, consider the strong solution to ~\eqref{eq:uld} given by $X_T(\omega) = [U_T(\omega), V_T(\omega)]$ for some $T > 0$. The algorithm $A$ to approximate $U_T(\omega)$ has oracle access to $(\nabla f(x), \int_{0}^{t}e^{\theta s}dB_t(\omega)) $ for any $x \in \sR^d$, $t \in [0,T]$ and $s\in \{0,2\}$ along with independent randomness $\tilde{\omega} \in \tilde{\Omega}$. The algorithm queries the oracle with $(x,t)$ of choice multiple times to produce an estimate $A(f,\omega,\tilde{\omega})$ for $U_T(\omega)$. This includes the case of Underdamped RLMC and Underdamped LMC. Their main result demonstrates that $\inf_{A \in \mathcal{A}_N}\sup_f E_{\omega,\tilde{\omega}}\|U_{T}(\omega)-A(f,\omega,\tilde{\omega})\|^2 \gtrsim C(T,L,\alpha)\frac{d}{N^3}$, where $\mathcal{A}_N$ is the set of all randomized algorithms as above with $N$ oracle queries. This error is the strong $L^2$ error since the algorithm and the SDE are driven by the same Brownian motion. This shows that algorithms of the class above need $N = \tilde{\Omega}_{\kappa} (\tfrac{1}{\epsilon^{2/3}})$ oracle queries to achieve strong $L^2$ error $\tfrac{\epsilon^2 d} {\alpha}$ and Underdamped RLMC achieves this optimal bound.

However, sampling algorithm guarantees consider `weak errors' which are distances between $\law(U_T(\omega))$ and $\law(A(f,\omega,\tilde{\omega}))$. In particular, the Wasserstein-2 distance is the infimum of $L^2$ errors when $U_T$ is driven by $B_t(\omega)$ and $A(\cdot)$ queries $B_t^{\prime}(\omega)$ over all couplings of \textit{distinct} Brownian motions $B_t(\omega)$ and $B'_t(\omega)$. Our results show Poisson ULMC queries the oracle $\tilde{\mathcal O}_{\kappa, d}(\frac{1}{\epsilon^{1/3}})$ times in expectation to achieve $\mathcal W_2^2(\law(A(f,\omega,\tilde{\omega})),\pi) \leq \frac{\epsilon^2 d}{\alpha}$, a quadratic improvement over RLMC. 

We note that Kandasamy \& Nagaraj~\cite{kandasamy2024poissonmidpointmethodlangevin} obtained a complexity upper bound of $\tilde {\mathcal O}_{d,\kappa}(\frac{1}{\sqrt{\epsilon}})$ for Underdamped PLMC under LSI assumptions for achieving $\mathrm{TV} \leq \epsilon$. The literature on sampling algorithms compares bounds of the form $\mathcal W_2^2 \leq \frac{\epsilon^2}{\alpha}$ to bounds of the form $\mathrm{TV} \leq \epsilon$ (see Section~\ref{sec:prior_work}). Under this comparison our bound improves over prior art. However, we note that $\mathrm{TV}$ and $\mathcal W_2^2$ bounds cannot be directly related rigorously. 
\section{Notation and problem setup}
\label{Problem_setup}
Let $\|\cdot\|$ denote the standard Euclidean norm over $\sR^d$ for some $d$ and $\vI_d$ denote the $d\times d$ identity matrix. The notation $x = \mathcal O(y)$ and $x \lesssim y$ mean there exists a universal constant $C > 0$ such that $x \leq Cy,$ and $\tilde {\mathcal O}(\cdot)$ hides logarithmic factors. The notation $\mathcal O_a(\cdot), \Omega_a(\cdot)$ mean the same as $\mathcal O(\cdot), \Omega(\cdot)$ except that they hide log factors. The number of evaluations of $\nabla f$ by the algorithm is referred to as `oracle complexity'.  We call the number of arithmetic operations (such as addition and multiplication) required on top of the oracle queries as `arithmetic complexity'. PLMC can be implemented such that $\text{arithmetic complexity} = O(d\times\text{oracle complexity})$ as shown in the sequel. Thus, as is common in the literature, we only report the oracle complexity guarantees. Let $\law(X)$ denote the law of the random variable $X$. Given two probability measures $\mu$ and $\nu$, we let $\KL(\mu||\nu)$ denote the KL divergence and $\mathrm{TV}(\mu,\nu)$ denote the total variation distance between them.

Given a sequence of probability measures $\mu_i$ over $\mathcal{X}_i$, for $i \in [n]$, a coupling is a probability measure $\Gamma$ over the product space $\prod_i \mathcal{X}_i$ such that the marginal over $\mathcal{X}_j$ is $\mu_j$. A sequence of random variables $(X_i \sim \mu_i)$ are coupled if they are defined over a common probability space, since their joint law is a coupling of $(\mu_i)_{i \in [n]}$. 
The Wasserstein-2 distance between $\mu$ and $\nu$ is given by
\small
\[\W{\mu}{\nu} := \inf_{\zeta \in \Gamma(\mu,\nu)}\int  ||x-y||^2 d\zeta(x,y),\]
\normalsize
where $\Gamma(\mu,\nu)$ denotes the set of couplings of $\mu$ and $\nu$. We make the following assumptions on $f$. 
\begin{assumption}\label{well_conditioned_F}
    The function $f: \mathbb{R}^d \to \mathbb{R}$ is $\alpha$ strongly convex and $L$ smooth for some $\alpha, L > 0$. That is, $f$ is twice continuously differentiable over $\sR^d$ and for every $x,y \in \sR^d$, we have: $ f(y)-f(x) \geq \langle \nabla f(x), y-x\rangle + \frac{\alpha}{2}\|x-y\|^2$ and
$\|\nabla f(x) - \nabla f(y)\| \leq L\|x-y\|$.
\end{assumption}
The target distribution, given by the density $\pi(x) \propto \exp(-f(x))$, is then called strongly log-concave. Our goal is to sample a random variable $X\sim \mu$ such that:\footnote{scaling $\frac{d}{\alpha}$ as considered by Shen \& Lee~\cite{shen2019randomizedmidpointmethodlogconcave}.}
\small
 \begin{equation}\label{eq:end_condition}
 \W{\mu}{\pi} \leq \frac{\epsilon^2 d}{\alpha}.
\end{equation}
\normalsize
We define the condition number $\kappa := \frac{L}{\alpha}$. Our notion of complexity is the number of gradient calls of $F$, in terms of the problem parameters $\kappa,d$ and $\epsilon.$ 
\subsection{Langevin Monte Carlo}
Suppose we wish to sample from $\pi \propto \exp(-f(x))$ in $\sR^d.$

\textbf{Overdamped LMC (OLMC)} with step-size $\eta$ is the discrete time algorithm defined by the following iterates: \small
\[X_{t+1} = X_t - \eta \nabla f(X_t) + \sqrt{2\eta}Z_t,\]\normalsize
where $Z_t \in \sR^d$ is an independent standard Gaussian. This is the Euler-Maruyama discretization of Overdamped Langevin dynamics (OLD): 
\small
\[dX_t = -\nabla f(X_t)dt + \sqrt 2 dB_t, \]
\normalsize
whose stationary distribution is $\pi.$ \citep{roberts1996exponential}

\textbf{Underdamped LMC (ULMC):}\label{ulmc_definition} Let $U_t \in \sR^d$ denote position, and $V_t \in \sR^d$ denote momentum. ULMC with step-size $\eta$ is defined via the following recursion:\small
\[\begin{bmatrix}
    U_{t+1} \\V_{t+1}
\end{bmatrix} = A(\eta)\begin{bmatrix}
    U_t \\ V_t
\end{bmatrix} - G(\eta)\begin{bmatrix}
    \nabla f(U_t) \\ 0
\end{bmatrix} + \Gamma(\eta) Z_t,\]\normalsize
where $Z_t \in \sR^{2d}$ is an independent standard Gaussian, and \small
\[A(\eta) = \begin{bmatrix}
    \vI_d & \tfrac{1}{\gamma}(1-e^{-\gamma \eta})\vI_d\\ 0 &
e^{-\gamma \eta}\vI_d
\end{bmatrix}, \text{   } G(\eta) = \begin{bmatrix} \tfrac{1}{\gamma}(\eta-\tfrac{1}{\gamma}(1-e^{-\gamma \eta}))\vI_d & 0\\ \tfrac{1}{\gamma}(1-e^{-\gamma \eta})\vI_d & 0 \end{bmatrix}, \]
\[\Gamma(\eta)^{2} := \begin{bmatrix} \frac{2}{\gamma}\left(\eta - \frac{2}{\gamma}(1-e^{-\gamma \eta}) + \frac{1}{2\gamma}(1-e^{-2\gamma \eta})\right) \vI_d & \tfrac{1}{\gamma}(1-2e^{-\gamma \eta} + e^{-2\gamma \eta})\vI_d \\ \tfrac{1}{\gamma}(1-2e^{-\gamma \eta}+ e^{-2\gamma \eta})\vI_d & (1-e^{-2\gamma \eta})\vI_d\end{bmatrix}. \] \normalsize This is the Euler-Maruyama discretization of the underdamped Langevin dynamics: 
\small
\begin{equation}\label{eq:uld} dU_t = V_t dt, \text{    } dV_t = -\gamma V_t -\nabla f(U_t)dt + \sqrt 2dB_t.
\end{equation}
\normalsize
The stationary distribution of these dynamics is $\pi(U,V) \propto \exp(-f(U)-\frac{1}{2}||V||^2).$ 
\citep{eberle2019couplings,dalalyan2020sampling}
\subsection{Poisson Midpoint Method}\label{plmc_definitions}
The Poisson midpoint method is a discrete variant of the randomized midpoint method introduced by \citet{shen2019randomizedmidpointmethodlogconcave}. The iterates of PLMC run in batches of size $k;$ and can be interpreted as a stochastic approximation of Langevin Monte-Carlo, with step-size $\eta/k$. Let $t$ and $i$ be integers, with $t\geq 0$ and $0\leq i\leq k-1.$ 

To emphasize the comparison with PLMC, we adopt the following notation for \textbf{overdamped LMC:} 
\small
\[X_{t,i+1} = X_{t,i} - \frac{\eta}{k}\nabla f(X_{t,i}) + \sqrt{\frac{2\eta}{k}}Y_{t,i},\]
\[X_{t+1,0} = X_{t,k}.\] 
\normalsize
Here $Y_{t,i} \in \sR^d$ denote independent standard Gaussians. Note that this is OLMC with step-size $\eta/k,$ grouped into batches of size $k$. Now let $Z_{t,i} \in \sR^d$ be independent standard Gaussians, and $H_{t,i}$ be independent Bernoulli random variables with parameter $1/k.$

\textbf{Overdamped PLMC} is defined by the following recursions:
\small
\[\tilde X_{t,i}^+ = \tilde X_{t,0} - \frac{\eta i}{k}\nabla f(\tilde X_{t,0}) + \sum_{j=0}^i\sqrt{\frac{2\eta}{k}}Z_{t,j}\]
\[\tilde X_{t,i+1} = \tilde X_{t,i} - \frac{\eta}{k}\nabla f(\tilde X_{t,0}) + \eta H_{t,i}(\nabla f(\tilde X_{t,0}) - \nabla f(\tilde X_{t,i}^+)) + \sqrt{\frac{2\eta}{k}}Z_{t,i}\]
\[\tilde X_{t+1,0} = \tilde X_{t,k}\]
\normalsize
\begin{remark}\label{plmc_explanation}
    The iterates $\tilde X_{t,i}^+$ denote midpoints. They are defined the same way as in \citet{shen2019randomizedmidpointmethodlogconcave}. The correction term $\eta H_{t,i}(\nabla f(\tilde X_{t,0}) - \nabla f(\tilde X_{t,i}^+))$ decides whether we use the gradient evaluated at our midpoint. In expectation over $H_{t,i},$ the drift term is $\ebk \nabla f(\txtki^+).$ However, we only need to evaluate $\nabla f(\txtki^+)$ when $H_{t,i} = 1.$ With $N_t = \sum_{i=0}^{k-1}H_{t,i}$ we have $\E N_t = 1.$ This means we need an expected number of 2 gradient calls to $f$ per batch including $\nabla f(\xt).$ This facilitates an efficient implementation of PLMC where $\tilde X_{t+1,0}$ can be computed directly from $\tilde X_{t,0},$ with $\tilde {\mathcal O}(1)$ gradient calls and an arithmetic complexity of $\tilde {\mathcal O}(d).$ This relies on the properties of jointly Gaussian random variables, and is detailed in \cite[Section 2.2]{kandasamy2024poissonmidpointmethodlangevin}. This is explicated to the case of overdamped PLMC in Algorithm \ref{efficient_implementation_overdamped_plmc}. \footnote{The original paper contains a typo, which has been rectified in our exposition.}
\end{remark}

\begin{algorithm}
\caption{Efficient Implementation of Overdamped PLMC Step.}\label{efficient_implementation_overdamped_plmc}
\textbf{Step 1.} Generate $\sI_t = \{i_1,\dots,i_{N_t}\}$ such that $H_{t,i} = 1$ if and only if $i \in \sI_t,$ and $i_1<\dots< i_{N_t}$ \;

\textbf{Step 2.} $m_{t,0} \leftarrow 0$, $Z_{t,n} \sim \mathcal{N}(0,\mathbf{I})$ i.i.d. $i_0 \leftarrow 0$, $i_{N_t+1} \leftarrow k-1$. For $1 \leq n \leq N_t+1$:
\small
\[m_{t,n} \leftarrow m_{t,n-1} + \sqrt{\frac{2\eta(i_n-i_{n-1})}{k}}Z_{t,n},\]
\normalsize

\textbf{Step 3.} For $1 \leq n \leq N_t,$
\small
$$\tilde X_{t,i_n}^+ \leftarrow \tilde X_{t,0} - \frac{\eta i_n}{k}\nabla f(\tilde X_{t,0}) + m_{t,n}.$$ 
\normalsize

\textbf{Step 4.}
\small
$$\Delta_t \leftarrow \ebk\sum_{n=1}^{N_t}(\nabla f(\xt) - \nabla f(\tilde X_{t,i_n}^+))$$
\normalsize

\textbf{Step 5.} 
\small 
$$\tilde X_{t+1,0}  \leftarrow \tilde X_{t,0} - \ebk \nabla f(\xt) + \Delta_t + m_{t,N_t+1}$$
\normalsize
\end{algorithm}

\textbf{Underdamped PLMC} is defined in a similar manner, by the following recursions: 
\small
\[\begin{bmatrix}
    \tilde U_{t,i}^+ \\ \tilde V_{t,i}^+
\end{bmatrix} = A\Big(\frac{\eta i}{k}\Big)\begin{bmatrix}
    \tilde U_{t,0} \\ \tilde V_{t,0}
\end{bmatrix} - G\Big(\frac{\eta i}{k}\Big)\begin{bmatrix}
    \nabla f(U_{t,0}) \\ 0
\end{bmatrix} + \sum_{j=0}^{i-1}A\Big(\frac{\eta(i-1-j)}{k}\Big)\Gamma\Big(\frac{\eta i}{k}\Big) Z_{t,i}\] 

\begin{align*}
\begin{bmatrix}
    \tilde U_{t,i+1} \\ \tilde V_{t,i+1}
\end{bmatrix} = A\Big(\frac{\eta}{k}\Big) \begin{bmatrix}
    \tilde U_{t,i} \\ \tilde V_{t,i}
\end{bmatrix} & - G\Big(\frac{\eta}{k}\Big)\begin{bmatrix}
    \nabla f(\tilde U_{t,0}) \\ 0
\end{bmatrix} + \Gamma\Big(\frac{\eta}{k}\Big) Z_{t,i} + k H_{t,i}\cdot G\Big(\frac{\eta}{k}\Big)
\begin{bmatrix}
    \nabla f(\tilde U_{t,0}) - \nabla f(\tilde U_{t,i}^+)\\ 0
\end{bmatrix}
\end{align*}

\[\begin{bmatrix}
\tilde U_{t+1,0} \\ \tilde V_{t+1,0}
\end{bmatrix} = \begin{bmatrix}
    \tilde U_{t,k} \\ \tilde V_{t,k}
\end{bmatrix}\]
\normalsize
With $A,G$ and $\Gamma$ as defined in \ref{ulmc_definition}, and $Z_{t,i}\in \sR^d$ being independent standard Gaussians. 
\begin{remark}
    As in the overdamped case, $\tilde U_{t,i}^+$ and $\tilde V_{t,i}^+$ denote midpoints, and the outcome of the Bernoulli decides whether we evaluate the gradient at the midpoint. We note that the comments on complexity in Remark \ref{plmc_explanation} are also valid in the underdamped case. An efficient implementation of underdamped PLMC is given in Algorithm \ref{efficient_implementation_underdamped_plmc}. 
\end{remark}
We adopt the following notation for \textbf{underdamped LMC}, to emphasize the comparison to PLMC. 
\small
\[\begin{bmatrix}
    U_{t,i+1} \\ V_{t,i+1}
\end{bmatrix} = A\Big(\frac{\eta}{k}\Big) \begin{bmatrix}
    U_{t,i} \\ V_{t,i}
\end{bmatrix} - G\Big(\frac{\eta}{k}\Big)\begin{bmatrix}
    \nabla f(U_{t,i}) \\ 0
\end{bmatrix} + \Gamma\Big(\frac{\eta}{k}\Big)Y_{t,i},\]
\[\begin{bmatrix}
    U_{t+1,0} \\ V_{t+1,0}
\end{bmatrix} = \begin{bmatrix}
    U_{t,k} \\ V_{t,k}
\end{bmatrix},\]
\normalsize
where $Y_{t,i} \in \sR^{2d}$ is an independent standard Gaussian. Note that this is underdamped LMC with step-size $\eta/k,$ grouped into batches of size $k.$ 
\subsection{Prior work}
\label{sec:prior_work}
\begin{table}[h]
  \caption{Complexity for discretized OLD. In case of LSI, $\kappa = L\times$\text{ LSI constant}. The scaling of $\mathcal W_2^2$ is different from~\eqref{eq:end_condition} to compare with TV and KL bounds.}
  \label{tab:comparison-overdamped}
  \resizebox{0.99\columnwidth}{!}{
  \centering
  \scriptsize 
  \begin{tabular}{|l|l|l|l|}
    \toprule
    \multicolumn{4}{|c|}{Overdamped Langevin Dynamics}               \\
    \midrule
    Algorithm    & Assumption  & Metric  & Oracle Complexity \\
    \midrule 
    LMC \cite{durmus2018analysislangevinmontecarlo} & Strongly Log-Concave & $\mathcal W_2^2 \leq \frac{\epsilon^2}{\alpha} $ & $\frac{\kappa d}{\epsilon^2}$ \\
    \midrule
    RLMC \cite{shen2019randomizedmidpointmethodlogconcave, yu2023langevinmontecarlostrongly} & Strongly Log-Concave & $\mathcal W_2^2 \leq \frac{\epsilon^2}{\alpha} $ & $\frac{\kappa \sqrt{d}}{\epsilon} + \frac{\kappa^{4/3}d^{1/3}}{\epsilon^{2/3}}$  \\
    \midrule
    RLMC \cite{altschuler2024shiftedcompositioniiilocal} & Strongly Log-Concave & $\mathsf{KL} \leq \epsilon^2$  & $\frac{\kappa \sqrt{d}}{\epsilon}$    \\
    \midrule
    RLMC \cite{altschuler2024shiftedcompositioniiilocal} & LSI & $\mathsf{KL} \leq \epsilon^2$  & $\frac{\kappa^{3/2} \sqrt{d}}{\epsilon}$    \\
    \midrule
    PLMC (Ours) & Strongly Log-Concave &  $\mathcal W_2^2 \leq \frac{\epsilon^2}{\alpha}$ & $\frac{\kappa^{4/3}d^{1/3}+\kappa d^{2/3}}{\epsilon^{2/3}}$    \\
    \bottomrule
    \end{tabular}
  }
\end{table}

\begin{table}[h]
  \caption{Complexity for discretized ULD. In case of LSI, $\kappa = L\times$\text{ LSI constant}. The scaling of $\mathcal W_2^2$ is different from~\eqref{eq:end_condition} to compare with TV and KL bounds, and $p \in \mathbb N$ is arbitrary.}
  \label{tab:comparison-underdamped}
  \resizebox{0.99\columnwidth}{!}{
  \centering
  \scriptsize
  \begin{tabular}{|l|l|l|l|}
    \toprule
    \multicolumn{4}{|c|}{Underdamped Langevin Dynamics}               \\
    \midrule
    Algorithm    & Assumption  & Metric  & Oracle Complexity \\
    \midrule 
    LMC \cite{dalalyan2020sampling} & Strongly Log-Concave & $\mathcal W_2^2 \leq \frac{\epsilon^2}{\alpha} $ & $\frac{\kappa^{3/2}\sqrt{d}}{\epsilon}$ \\
    \midrule
    RLMC \cite{shen2019randomizedmidpointmethodlogconcave, yu2023langevinmontecarlostrongly} & Strongly Log-Concave & $\mathcal W_2^2 \leq \frac{\epsilon^2}{\alpha} $ & $\frac{\kappa d^{1/3}}{\epsilon^{2/3}} + \frac{\kappa^{7/6}d^{1/6}}{\epsilon^{1/3}}$  \\
    \midrule
    PLMC \cite{kandasamy2024poissonmidpointmethodlangevin} & LSI &  $\mathsf{TV} \leq \epsilon$ & $\frac{\kappa^{\frac{17}{12}}d^{\frac{5}{12}}}{\sqrt{\epsilon}}$  \\
    \midrule
    PLMC (Ours) & Strongly Log-Concave &  $\mathcal W_2^2 \leq \frac{\epsilon^2}{\alpha}$ & $\frac{\kappa^{7/6}d^{1/3}}{\epsilon^{1/3}}+ \frac{\kappa^{\tfrac{11p+6}{8p+6}}d^{\tfrac{3p+2}{8p+6}}}{\epsilon^{\tfrac{p+2}{4p+3}}}$ \\
    \bottomrule
  \end{tabular}
  }
\end{table}

Recent works have focused on the rigorous theoretical analysis of classical and popular MCMC algorithms to establish complexity bounds and theoretical limits. The prototypical case of Overdamped LMC has been studied when the target $\pi$ is strongly log-concave
and more generally when $\pi$ satisfies isoperimetric inequalities \citep{dalalyan2017theoretical,durmus2017unadjusted,durmus2018analysislangevinmontecarlo,vempala2022rapidconvergenceunadjustedlangevin,erdogdu2021-tail-growth,mou2022improved,sinho-non-log-concave-lmc}. Underdamped LMC has been considered as a faster alternative. This case too has been well studied when $\pi$ is strongly log-concave and when $\pi$ satisfies isoperimetric inequalities \citep{cheng2018-underdamped-lmc,dalalyan2020sampling,ganesh2020faster, ma2021there,zhang2023improveddiscretizationanalysisunderdamped,altschuler2025shiftedcompositionivunderdamped}

LMC is the Euler-Maruyama discretization of continuous time Langevin Dynamics, which can lead to sub-optimal convergence due to statistical bias in the approximation. Thus, Shen \& Lee~\cite{shen2019randomizedmidpointmethodlogconcave} introduced the randomized midpoint method for LMC (RLMC) which reduces the bias in the approximation by introducing a randomized estimator at the cost of higher variance. RLMC does not involve higher order derivatives of $\nabla f$ as in Runge-Kutta schemes for SDEs \citep{kloeden1992stochastic} - allowing its use for generative modeling with denoising diffusion models \citep{kandasamy2024poissonmidpointmethodlangevin}. This leads to improvement in the convergence rates compared to LMC under log concavity (see Tables~\ref{tab:comparison-overdamped} and \ref{tab:comparison-underdamped}). The works \cite{he2021ergodicitybiasasymptoticnormality,yu2023langevinmontecarlostrongly,altschuler2024shiftedcompositioniiilocal,altschuler2025shiftedcompositionivunderdamped} extend the bounds in \cite{shen2019randomizedmidpointmethodlogconcave}.

Kandasamy \& Nagaraj~\cite{kandasamy2024poissonmidpointmethodlangevin} introduced the Poisson midpoint method for LMC (PLMC), a variant of RLMC, which converges whenever LMC converges, allowing analysis beyond log-concavity. PLMC gives a quadratic improvement in complexity in terms of $\epsilon$ when $\pi$ satisfies Logarithmic Sobolev Inequalities (LSI). Our work shows a cubic improvement for PLMC under strong log-concavity.

The literature on MCMC considers various notions of convergence including KL-divergence, TV and $\mathcal{W}_2$. In the case when $\pi$ is strongly log-concave, the Otto-Villani Theorem \citep{otto2000generalization} shows that $\KL(\mu||\pi) \leq \epsilon^2 \implies \W{\mu}{\pi} \lesssim \frac{\epsilon^2}{\alpha}$ and the Pinsker's inequality shows that $\KL(\mu||\pi) \leq \epsilon^2 \implies \mathsf{TV}(\mu,\pi) \lesssim \epsilon$. The condition of $\pi$ satisfying LSI is more general than strong log-concavity of the target. We refer to Tables~\ref{tab:comparison-overdamped} and \ref{tab:comparison-underdamped} for a detailed comparison of the results.

\section{Results}
We now present our main results. The following Theorem on the convergence of overdamped PLMC is proven in Section \ref{olmc_main_proof}. 
\begin{theorem}\label{OLMCmain}
    Let $\tilde X_{t,i}$ denote the iterates of Overdamped PLMC, and $X_{t,i}$ the iterates of Overdamped LMC with stepsize $\eta/k,$ as defined in Section \ref{plmc_definitions}. Assume $\eta L \leq 1/8,$ and Assumption \ref{well_conditioned_F}. Then there exist absolute constants $c_1$ and $c_2$ such that
    \small
    \begin{align*}
        \W{\law(\xt)}{\law(X_{t,0})} \lesssim & (\eta^6 L^4 dk + \eta^4L^2 + \tfrac{\eta^5 L^4}{\alpha})\cdot(Ldt + \tfrac{1}{\eta}\e(f(X_{0,0})-f(X_{t,0})) \\ & + \tfrac{\eta^3L^4d}{\alpha^2}+ \tfrac{\eta^4 L^4 d^2}{\alpha}+\exp(c_1d-(c_2 \eta^2 L^2 k)^{-1})\cdot \tfrac{\eta^2 L^2d}{\alpha}.
    \end{align*}
\normalsize
\end{theorem} 
The above theorem shows that $\tilde X_{t,0}$ is close to $X_{t,0}$ in Wasserstein-2 distance. However, running $tk$ iterations of PLMC requires only $\mathcal O(t)$ gradient calls, as compared to $tk$ gradient calls for LMC. In the following corollary, we combine the Theorem~\ref{OLMCmain} with the convergence results for $X_{t,i}$ to $\pi$ given in \cite{durmus2018analysislangevinmontecarlo} to deduce the convergence of $\tilde X_{t,i}$. We refer to Section~\ref{sec:OLMCcomplexityproof} for its proof.

\begin{corollary}\label{OLMCcomplexity}
Let $\tilde{X}_{t,0}$ be the iterates of Overdamped PLMC as in Theorem~\ref{OLMCmain}. Let $x^*$ be the unique minimizer of $f$, and $\epsilon>0.$ Assume: 
\begin{enumerate}
\item The conditions from Theorem \ref{OLMCmain} hold.
 \item $\tilde{X}_{0,0}$ satisfies $\e[f(\tilde{X}_{0,0})-f(x^*)] \leq C_f\kappa d$ for some $C_f > 0$. 
\end{enumerate}

Then there exist constants $C_1, C_2 > 0$ depending only on $C_f$, $\log(\frac{W_2(X_{0,0},\pi)\sqrt \alpha}{\epsilon \sqrt d})$ and $\log (1/\epsilon),$ polynomially, such that if $\eta = C_1\min(\frac{\alpha^{1/3}\epsilon^{2/3}}{L^{4/3}}, \frac{\epsilon^{2/3}}{d^{1/3}L}),$ $k \asymp \max(\frac{\eta L}{\epsilon^2},1)$  and $N = C_2 \big[\frac{\kappa^{4/3}+\kappa d^{1/3}}{\epsilon^{2/3}}\big]$. Then, 
\small
    $$\W{\law(\tilde X_{N,0})}{\pi} \leq \epsilon^2 d/\alpha$$
\normalsize
\end{corollary} 

\begin{remark}
The complexity bound for Overdamped LMC \citep{durmus2018analysislangevinmontecarlo} is $\tilde {\mathcal O}(\kappa/\epsilon^2)$ gradient calls, and that of Overdamped RLMC \citep{yu2023langevinmontecarlostrongly} is $\tilde {\mathcal O}(\frac{\kappa}{\epsilon}+\frac{\kappa^{4/3}}{\epsilon^{2/3}})$ gradient calls. To our knowledge, our method is thus the best known discretization of overdamped Langevin dynamics, in terms of $\epsilon$ dependence. Note that our assumption on the initialization is very mild - $f$ can be optimized easily using standard convex optimization algorithms.
\end{remark}

 The following Theorem, proved in Section \ref{ulmc_main_proof}, considers Underdamped Langevin Dynamics:
\begin{theorem}\label{ULMCmain}
    Let $\tilde U_{t,i}$ denote the iterates of Underdamped PLMC, and $U_{t,i}$ denote the iterates of Underdamped LMC with step-size $\eta/k,$ as defined in Section \ref{plmc_definitions}. Let $p \geq 0$ be any integer. There exists $c_0 >0$, which depends only on $p$ such that if:
    \begin{enumerate}
        \item Assumption \ref{well_conditioned_F} holds.
        \item $\gamma\eta < c_0,$ $\frac{\eta}{k} \leq \frac{c_0}{\kappa\sqrt{L}},$ and $\frac{\eta^{3p-1}t^{p-1}L^{2p}}{\gamma^{p+1}}<c_0$
        \item  $\gamma = c_{\gamma}\sqrt L$ for some constant $c_{\gamma} \geq \sqrt 2$. 
    \end{enumerate}
    \small
\begin{align*}
\text{Then,} \quad \quad    \W{\law(\tutki)}{\law(U_{t,i})} &=
    \mathcal O\Big[\tfrac{\eta^7 L^{9/2}d}{\alpha\gamma^2}t+\tfrac{\eta^8 L^4 d^2}{\gamma^2}t^2+\tfrac{\eta^{4p+4}k^{p-1}L^{2p+2}d^{p+1}}{\gamma^{2}}t^{p+1}\Big] \\ & + \e[P_\eta(||V_{0,0}||, |f(\Psi_0)-f(\Psi_t)|^+)],
\end{align*}
\normalsize
Where $\mathcal{O}$ hides constants depending only on $c_0, c_{\gamma}$. $P_\eta$ is a polynomial whose coefficients are high powers of $\eta$ and depend on $p, c_{\gamma}$, and $\Psi$ is defined as $\Psi_s := \tilde U_{s,0} + \frac{1}{\gamma}\tilde V_{s,0}.$ The complete bound is explicated in Section \ref{ULMCmain_complete}, for the sake of clarity.
\end{theorem}
The bound in Theorem \ref{ULMCmain} holds for any choice of nonnegative integer $p$. The presence of $p$ is due to the manner in which we bound a certain low probability event - see the proof of Proposition \ref{ULMC_coupling}. Similar to Corollary~\ref{OLMCcomplexity}, the following Corollary (proved in Section~\ref{sec:ULMCcomplexity_proof}) establishes complexity bounds. 
\begin{corollary}\label{ULMCcomplexity}
    Let $\tutki$ denote the iterates of Underdamped PLMC, as in Theorem \ref{ULMCmain}. Let $x^*$ be the unique minimizer of $f,$ and $p \in \sN \cup \{0\}$ be fixed. Let $k \asymp \max(\lceil\frac{\eta L}{\epsilon\sqrt\alpha}\rceil,1)$, and $\gamma = c_\gamma \sqrt L$ as in Theorem \ref{ULMCmain}. Initialize the iterates with $V_{0,0} \sim \mathcal N(0,\vI_d)$ and $\e||U_{0,0}-x^*||^{2n} \leq  c_1d^{n}/L^n$ for $n = \max(2, p+1),$ and some constant $c_1>0$ depending only on $p$. 

    Then there exist $C_3, C_4,C_5>0$ depending on $p$ and polynomially on $\log(\frac{\mathcal W_2(U_0,\pi)\sqrt \alpha}{\epsilon \sqrt d})$ and $\log(\frac{\kappa}{\epsilon}),$ such that: if $\eta = C_3\min\left(\frac{\epsilon^{1/3}}{\kappa^{1/6} d^{1/6}\sqrt L},\frac{\epsilon^{\frac{p+2}{4p+3}}}{\kappa^{\frac{3p}{8p+6}}d^{\frac{p}{4p+3}}\sqrt L}\right),$ $0<\epsilon \leq C_4\min(\kappa^{-1/2}, \kappa^{-\frac{2p-3}{2p}}d^{1/2})$ and $N = C_5 \big[\frac{\kappa^{7/6}d^{1/6}}{\epsilon^{1/3}}+\frac{\kappa^{\frac{11p+6}{8p+6}}d^{\frac{p}{4p+3}}}{\epsilon^{\frac{p+2}{4p+3}}}\big]$, we have 
    \small 
    $$\W{\law(\tilde U_{N,0})}{\pi} \leq \epsilon^2 d/\alpha\,.$$
    \normalsize
\end{corollary}
The complexity bound for Underdamped LMC is $\tilde {\mathcal O}(\kappa^{3/2}/\epsilon)$ \citep{cheng2018-underdamped-lmc}, and that of Underdamped RLMC is $\tilde {\mathcal O}(\frac{\kappa^{7/6}}{\epsilon^{1/3}}+\frac{\kappa}{\epsilon^{2/3}})$ \citep{shen2019randomizedmidpointmethodlogconcave}. 
\begin{remark}
Our assumption on the initialization is standard in the literature \citep{vempala2022rapidconvergenceunadjustedlangevin,shen2019randomizedmidpointmethodlogconcave}, and satisfied (for example) by $\mathcal N(x^*, \vI_d/L).$ 
\begin{enumerate}
\item With $p=0$, we get a complexity of $\tilde {\mathcal O}(\frac{\kappa^{7/6}d^{1/6}}{\epsilon^{1/3}}+ \frac{\kappa}{\epsilon^{2/3}}).$
\item With $p=3$, we get a complexity of $\tilde {\mathcal O}(\frac{\kappa^{13/10}d^{1/5}}{\epsilon^{1/3}}).$
\item For $p>3,$ the second term becomes lower order in $\epsilon$ and the oracle complexity satisfies $\tilde{\mathcal O}_{\kappa, d}(\frac{1}{\epsilon^{1/3}} + \frac{1}{\epsilon^{1/4 + \mathcal O(1/p)}})$.  
\end{enumerate}

\end{remark}
\begin{remark}
    The concurrent work of \cite{altschuler2025shiftedcompositionivunderdamped} claims an oracle complexity of $\tilde {\mathcal{O}}(\kappa^{5/6}d^{5/3}/\epsilon^{2/3})$ to achieve $\text{KL} \leq \epsilon^2.$ This is in the low friction regime $\gamma \asymp \sqrt\alpha,$ and for a double midpoint implementation of Underdamped RLMC. This has improved dependence in $\kappa$ as compared to prior works, but is worse in $d$ and without improvement in $\epsilon.$ 
\end{remark}
Our work improves dependence in $\epsilon$ while being worse in $d$. To our knowledge, PLMC is the best known discretization of underdamped Langevin dynamics in terms of $\epsilon,$ and is the first known algorithm to break the $\tilde {\mathcal O}(\epsilon^{-2/3})$ barrier for strongly log-concave sampling. 
\section{Intuition and Proof Idea} 
Our proof relies on the following key Lemma. This is similar to Lemma 7 of \cite{kandasamy2024poissonmidpointmethodlangevin}, which was in turn adapted from \cite{zhai2017highdimensionalcltmathcalw2distance}. The difference is that our result avoids higher order moments, making it significantly easier to apply. 
\begin{lemma}\label{new_CLT}
Let $V$ be a random vector in $\sR^d$ satisfying the following conditions: 
    \begin{enumerate}
        \item $||V|| \leq \beta$ a.s., $\e[V] = 0$, and $\e[VV^T] = \Sigma.$ 
        \item $V$ lies in a one-dimensional subspace almost surely. 
    \end{enumerate}
    Let the random vector $Z \sim \mathcal{N}(0,\vI_d),$ and independent of $V$. Let $\nu = \text{Tr}(\Sigma),$ Then,
    \small
    \[\W{\law(Z)}{\law(Z+V)} \leq \frac{11}{2}\nu^2+{\mathbf 1}_{5\beta^2 > 1} \cdot 2\nu.\]
    \normalsize
\end{lemma}
A naive bound would be $\W{\law(Z)}{\law(Z+V)} \leq \nu,$ which corresponds to the Gaussians being coupled identically. Note that $\nu^2$ can be much smaller than $\nu,$ and this leads to our sharp result. 

\textbf{Interpreting overdamped PLMC as LMC with perturbed Gaussian noise.} From the definition in Section \ref{plmc_definitions}, overdamped PLMC can be written as follows. \small
\[\tilde X_{t,i+1} = \txtki - \ebk \nabla f(\txtki) + \sqrt{\frac{2\eta}{k}}\tilde Z_{t,i},\]
\normalsize
where $\zt$ denotes the perturbed Gaussian and is given by the following expression. \small\[\zt =  \sqrt{\frac{\eta k}{2}}(H_{t,i}-1/k)(\nabla f(\tilde X_{t,0}) - \nabla f(\xpi)) + \sqrt{\frac{\eta}{2k}}(\nabla f(\xp) - \nabla f(\xpi)) + \z.\] \normalsize
Conditioned on the previous iterates $\xt,\xpi$ and $\txtki,$ this is a Gaussian with mean $B_{t,i} = \sqrt{\frac{\eta}{2k}}(\nabla f(\xp) - \nabla f(\xpi)),$ perturbed by the zero-mean random vector $S_{t,i} = \sqrt{\frac{\eta k}{2}}(H_{t,i}-1/k)(\nabla f(\tilde X_{t,0}) - \nabla f(\xpi)).$ Note that $S_{t,i}$ lies in a one dimensional subspace a.s., since it is determined by the Bernoulli $(\bern).$

\textbf{Gradient descent is contractive.} Given $\eta <1,$ and that $f$ is well-conditioned (Assumption \ref{well_conditioned_F}), the map $T(x) = x - \eta\nabla f(x)$ is Lipschitz with parameter $(1-\alpha\eta).$ 

\textbf{Constructing a coupling.} As seen in Section \ref{plmc_definitions}, iterates of Langevin Monte-Carlo are defined by \small\[X_{t,i+1} = X_{t,i} - \frac{\eta}{k}\nabla f(X_{t,i}) + \sqrt{\frac{2\eta}{k}}Y_{t,i}.\] \normalsize
In order to couple $X_{t,i+1}$ and $\tilde X_{t,i+1},$ we first let $X_{t,i}$ and $\tilde X_{t,i}$ be coupled optimally. Conditioned on $X_{t,i}, \txtki, \xpi$ and $\xt,$ we couple $\y$ and $\zt$ optimally as per the bound established in Lemma \ref{new_CLT}. This allows us to produce a recursion of the following form. \small
\[\W{\law(X_{t,i+1})}{\law(\tilde X_{t,i+1})} \leq (1-\frac{\alpha\eta}{2k})\W{\law(X_{t,i})}{\law(\tilde X_{t,i})} + E_{t,i},\] \normalsize
where $E_{t,i}$ is an appropriate bound on the discretization error.

\textbf{Bounding the discretization error.} The application of the CLT as detailed above gives us terms of the form $\e||\txtki-\xt||^p$ and $\e||\txtki^+-\xt||^p$ for some $p \in \mathbb{N}$. These can be bounded in terms of $\e||\nabla f(\xt)||^p$ and Gaussian moments. We then reduce the bounds to $\e||\nabla f(\xt)||^2$ rather than $\e||\nabla f(\xt)||^p$, and then apply the following gradient bound, which we believe is tight. 
\begin{lemma}\label{OLMCgradient} Assuming $\eta L \leq 1/8,$ the following bound is true. \small
\[\sum_{t=0}^{N-1} \e\|\gradf{\xt}\|^2 \lesssim \frac{1}{\eta}\e[f(\tilde X_{0,0})-f(\tilde X_{N,0})] + {L d N}.\]\normalsize
\end{lemma}
This is proven in Section \ref{OLMCgradient_proof}. It is known \citep[Lemma 11]{vempala2022rapidconvergenceunadjustedlangevin} that 
  $\int_{\sR^d} \|\nabla f(x)\|^2 d\pi(x) \leq Ld$ under smoothness.
 This bound is tight when $\pi$ is Gaussian. Therefore, we expect that the dominant term $LdN$ in our bound cannot be improved at stationarity. 

\textbf{The underdamped case.}
We make the following coordinate change for the iterates of underdamped LMC/PLMC. 
\small
\[\begin{bmatrix}
    x\\y
\end{bmatrix} \to \mathcal{M} \begin{bmatrix}
    x\\y
\end{bmatrix}, \text{   where   } \mathcal M = \begin{bmatrix}
    \vI_d & 0 \\ \vI_d & \frac{2}{\gamma}\vI_d
\end{bmatrix}.\]
\normalsize
Under this transformation, and with appropriate step-size; the deterministic component of the ULMC recursion is contractive. For a precise statement, see Lemma 16 of Zhang et al.~\cite{zhang2023improveddiscretizationanalysisunderdamped}. We denote $W_{t,i} = U_{t,i} + \frac{2}{\gamma}V_{t,i},$ and $X_{t,i} = [U_{t,i}, W_{t,i}]^T.$

Under our transformation $\mathcal M$, for appropriate matrices $\am, \gm, \Gamma_{\mathcal{M}}$ defined in Section~\ref{ulmc_main_proof}, we have:
\small
\[X_{t,i+1} = \am\Big(\frac{\eta}{k}\Big) \begin{bmatrix}
    U_{t,i} \\ W_{t,i}
\end{bmatrix} - \gm\Big(\frac{\eta}{k}\Big)\begin{bmatrix}
    \nabla f(U_{t,i}) \\ 0
\end{bmatrix} + \Gamma_{\mathcal{M}}\Big(\frac{\eta}{k}\Big)Y_{t,i},\]
\[X_{t+1,0}
= X_{t,k}
\]
\normalsize

This allows the ULMC recursion to be interpreted as a noisy contraction similar to OLMC. Define $T:\sR^{2d} \to \sR^{2d}$ by
\small
\[T\begin{bmatrix}
    u \\ w
\end{bmatrix}= \am(\eta) \begin{bmatrix}
    u \\ w
\end{bmatrix} - \gm(\eta) \begin{bmatrix}
    \nabla f(u) \\ 0
\end{bmatrix}.\] \normalsize
Then $T$ is Lipschitz with constant $(1-\frac{\alpha\eta}{\gamma}+L\eta^2)$ \citep[Lemma 16]{zhang2023improveddiscretizationanalysisunderdamped}, and is hence contractive for small $\eta.$ Using this perspective, we are able to follow a similar proof technique as in the overdamped case. In this case, we require bounds on the moments $\e||\nabla f(\utkt)||^p$ and $\e||\vtkt||^p.$ We use Theorem \ref{ULMCgradients}, to bound these moments. 
\section{Conclusion:}

We considered the Poisson Midpoint discretization of Overdamped and Underdamped Langevin Dynamics, and showed state of the art oracle complexity of $\tilde{\mathcal O}_{\kappa,d}(\frac{1}{\epsilon^{1/3}})$ for convergence in the Wasserstein-2 distance to the strong log-concave stationary law $\pi$. This breaks the conjectured lower bound of $\tilde{ \Omega}_{\kappa,d}(\frac{1}{\epsilon^{2/3}})$. Our work is an effort towards understanding the fundamental computational complexity of sampling from strongly log-concave distributions in terms of $\kappa, \epsilon$ and $d$, and shows an improved bound in terms of $\epsilon$. Concurrent work \citep{altschuler2025shiftedcompositionivunderdamped} claims an improvement of the state of the art dependence on $\kappa$ (from $\kappa^{7/6} \to \kappa^{5/6}$) but with a worse dependence on $\epsilon, d$. In future, we hope to explore techniques which simultaneously improve dependence on all three parameters. In particular, we believe our result can be improved in $\kappa$ if we obtain tight bounds on the moments $\e||\nabla f(\utkt)||^{p}$ and $\e||\vtkt||^p$ (Remark \ref{suboptimality_ulmc}), and this is an avenue for future research. 

\section*{Acknowledgment:}
We thank Sinho Chewi for pointing us to \cite{Cao_2021}, which presents lower bounds for the complexity of sampling via discretization of Underdamped Langevin Dynamics.

\newpage
\bibliography{citations}
\bibliographystyle{plain}

\appendix
\pagebreak
\section{Efficient Implementation of Underdamped PLMC}
\begin{algorithm}
\caption{Efficient Implementation of Underdamped PLMC Step.}\label{efficient_implementation_underdamped_plmc}
\textbf{Step 1.} Generate $\sI_t = \{i_1,\dots,i_{N_t}\}$ such that $H_{t,i} = 1$ iff $i \in \sI_t,$ and $i_1<\dots< i_{N_t}$ without loss of generality. \;

\textbf{Step 2.} Let $m_{t,0} \leftarrow 0,$ and $Z_{t,n} \in \sR^{2d}$ be a sequence of i.i.d. standard Gaussians. For $1 \leq n \leq N_t+1$: 
\[m_{t,n} \leftarrow A\Big(\frac{\eta (i_n-i_{n-1})}{k}\Big)m_{t,n-1} + \Gamma\Big(\frac{\eta (i_n-i_{n-1})}{k}\Big)Z_{t,n},\] \;
with the convention that $i_0 = 0$ and $i_{N_t+1} = k-1.$ \;

\textbf{Step 3.} For $1 \leq n \leq N_t,$ compute 
\[\begin{bmatrix}
    \tilde U_{t,i_n}^+ \\ \tilde V_{t,i_n}^+
\end{bmatrix} \leftarrow A\Big(\frac{\eta i_n}{k}\Big)\begin{bmatrix}
    \utkt \\ \vtkt
\end{bmatrix} - G\Big(\frac{\eta i_n}{k}\Big) \begin{bmatrix}
    \nabla f(\utkt)\\0
\end{bmatrix} + m_{t,n}.\] \;

\textbf{Step 4.} Compute the correction term: 
\[\Delta_t \leftarrow k \sum_{n = 1}^{N_t} A\Big(\frac{\eta(k-1-i_n)}{k}\Big)G\Big(\ebk\Big)\begin{bmatrix}
    \nabla f(\utkt) - \nabla f(\tilde U_{t,i_n}^+) \\ 0
\end{bmatrix}\] \; 

\textbf{Step 5.} Compute $\tilde U_{t+1,0}$ and $\tilde V_{t+1,0}:$
\[\begin{bmatrix}
    \tilde U_{t+1,0} \\ \tilde V_{t+1,0}
\end{bmatrix} \leftarrow A(\eta)\begin{bmatrix}
    \utkt \\ \vtkt
\end{bmatrix} - G(\eta) \begin{bmatrix}
    \nabla f(\utkt) \\ 0
\end{bmatrix} + \Delta_t + m_{t,N_t+1}\]
\end{algorithm}
\section{Proof of Lemma \ref{new_CLT}}
By the triangle inequality for $\mathcal W_2,$ we have 
\begin{align*}
    \W{\law(Z)}{\law(Z+V)}& \leq 2\W{\law(Z)}{\law(\sqrt{\vI_d+\Sigma}Z)}\\&+2\W{\law(\sqrt{\vI_d+\Sigma}Z)}{\law(Z+V)}
\end{align*}
The latter term is a Wasserstein distance between Gaussians, which has the following closed form. \[2\W{\sqrt{\vI_d+\Sigma}W}{Z} = 4+2\nu - 4\sqrt{1+\nu} \leq \frac{1}{2}\nu^2.\] 

The former term is bounded below (Lemma \ref{specializedCLT}), using a key result due to Alex Zhai. We check that the proof of \cite[Lemma 1.6]{zhai2017highdimensionalcltmathcalw2distance} does not require $n$ to be an integer and state the following:
\begin{lemma}[Lemma 1.6, \cite{zhai2017highdimensionalcltmathcalw2distance}]\label{Zhai} Let $n >0$ and let $Y$ be an $\mathbb R^k$ valued random variable with mean $0$, covariance $\Sigma/n$ and $\|Y\| \leq \frac{\beta}{\sqrt n}$ almost surely. For $t \geq 0$, let $Z_t$ denote a Gaussian of mean 0 and covariance $t\Sigma$ independent of $Y$. Let $\sigma_\text{min}^2$ denote the smallest eigenvalue of $\Sigma$. Then, for any $n \geq \frac{5\beta^2}{\sigma_\text{min}^2},$ we have \[\mathcal W_2(Z_1, Z_{1-1/n}+Y) \leq \frac{5\sqrt k \beta}{n\sqrt n}.\]
\end{lemma}

We note that the following Lemma is similar in form and proof to Lemma 7 of \cite{kandasamy2024poissonmidpointmethodlangevin}. 
\begin{lemma}\label{specializedCLT}
    Let $V$ be a random vector in $\sR^d$ satisfying the following conditions: 
    \begin{enumerate}
        \item $\|V\| \leq \beta$ a.s., $\e[V] = 0$, and $\e[VV^T] = \Sigma.$ 
        \item $V$ lies in a one-dimensional subspace almost surely. 
    \end{enumerate}
    Suppose the random vector $Z$ is distributed as $\mathcal{N}(0,\vI_d),$ and independent of $V$. Let $\nu = \text{Tr}(\Sigma),$ Then, $$\W{\law(\sqrt{\vI_d+\Sigma}Z)}{\law(Z+V)} \leq 5\nu^2+{\mathbf 1}_{5\beta^2 > 1} \cdot 2\nu.$$ 
\end{lemma}
\begin{proof}
    These distributions are the same along all directions perpendicular to $V.$ We couple those directions identically. Let $V'$ denote the projection of $V$ onto the direction spanned by itself, and $Z'$ denote a one-dimensional Gaussian. We get 
    \begin{align*}
        \mathcal W_2(\text{Law}(\sqrt{\vI_d+\Sigma}Z),\text{Law}(Z+V)) & \leq \mathcal W_2({\text{Law}(\sqrt{1+\nu}Z')},{\text{Law}(Z'+V')}) \\
        & = \sqrt{1+\nu}\mathcal W_2({\text{Law}(Z')},{\text{Law}(\tfrac{Z'}{\sqrt{1+\nu}}+\tfrac{V'}{\sqrt{1+\nu}}})).
        \intertext{Now set $k=1,$ $n=1+\frac{1}{\nu},$ and $\beta \to \beta\sqrt n$. Here $\sigma_{\text{min}} = 1,$ which means $5\beta^2 \leq 1$ is sufficient to apply Lemma \ref{Zhai}. }
        \mathbf 1_{5\beta^2\leq 1} \cdot \W{\text{Law}(\sqrt{\vI_d+\Sigma}Z)}{\text{Law}(Z+V)} & \leq \mathbf 1_{5\beta^2\leq 1}\cdot \frac{25\beta^2\nu^2}{1+\nu} \leq 5\nu^2. \\
        \intertext{When $5\beta^2 > 1,$ we couple $\law(\sqrt{1+\nu}Z')$ and $\law(Z'+V')$ to have the same Gaussian noise $Z'$, with $V'$ sampled independently of $Z'$. A simple computation yields}
        \mathbf 1_{5\beta^2 > 1} \cdot \W{\text{Law}(\sqrt{1+\nu}Z')}{\text{Law}(Z'+V')} & \leq \mathbf 1_{5\beta^2 > 1} \cdot 2\nu. 
    \end{align*}
\end{proof}
\section{Proof for Overdamped PLMC}\label{olmc_main_proof}
Recall from Section \ref{plmc_definitions} that $X_{t,i}$ denote the iterates of overdamped Langevin Monte Carlo with step-size $\ebk$. Similarly $\txtki$ denote the iterates of Poisson overdamped Langevin Monte Carlo with step size $\ebk$, and $\xpi$ denote midpoints. 
\begin{align*}
    X_{t,i+1} & = \x - \frac{\eta}{k} \nabla f(\x) + \sqrt{\frac{2\eta}{k}}\y  \\
    \tilde X_{t,i+1} & = \xp - \frac{\eta}
    {k} \nabla f(\xt) + \eta H_{t,i} (\nabla f(\xt) - \nabla f(\xpi))  + \sqrt{\frac{2\eta}{k}} \z \\
    \xpi & = \xt - \frac{\eta i}{k} \nabla f(\xt) + \sqrt{\frac{2\eta}{k}}\sum_{j = 0}^{i} Z_{t,j}
\end{align*}
The sequences $\z$ and $\y$ are i.i.d. standard Gaussians, and $H_{t,i}$ are independent Bernoullis with parameter $1/k$. All random variables above live on the same probability space, with a coupling we will specify. To interpret PLMC as LMC with a perturbed noise, we write 

\[\tilde X_{t,i+1} = \txtki - \ebk \nabla f(\txtki) + \sqrt{\frac{2\eta}{k}}\tilde Z_{t,i},\]
where $\zt$ denotes the perturbed Gaussian and is given by the following expression. \[\zt =  \sqrt{\frac{\eta k}{2}}(H_{t,i}-1/k)(\nabla f(\tilde X_{t,0}) - \nabla f(\xpi)) + \sqrt{\frac{\eta}{2k}}(\nabla f(\xp) - \nabla f(\xpi)) + \z.\] 

Let $B_{t,i} = \sqrt{\frac{\eta}{2k}}(\nabla f(\xp) - \nabla f(\xpi)),$ and $S_{t,i} = \sqrt{\frac{\eta k}{2}}(H_{t,i}-1/k)(\nabla f(\tilde X_{t,0}) - \nabla f(\xpi)).$ We refer to these as the bias and variance terms respectively. 

Define the event:
\[\g = \{\tilde X_{t,0} = y_0, \xp = y, \xpi = y^+, \x = x\},\]
with $x,y,y^+$ and $y_0$ being arbitrary points in $\mathbb R^d.$   For any valid coupling of $X_{t,i+1}$ and $\tilde X_{t,i+1}$ conditioned on $\g,$ the following holds. 
\begin{proposition}\label{OLMCamgm}
     Let Assumption \ref{well_conditioned_F} hold and let $\frac{\alpha\eta}{k} < 1$. Then we have,
     \[\e[ ||X_{t,i+1} - \tilde X_{t,i+1}||^2 | \g] \leq (1-\frac{\alpha\eta}{2k})^2 ||x-y||^2 + \frac{9\eta L^2}{\alpha k}||y- y^+||^2 + \frac{2\eta}{k} \e[||\z+S_{t,i}- \y||^2|\g].\]
\end{proposition}
The proof of this Proposition is in Section \ref{OLMCamgm_proof}. The first term arises from the contractivity of gradient descent, while the second term comes from the bias. We apply Lemma \ref{new_CLT} to bound the final term. 
\begin{corollary}\label{conditionalCoupling} Let $\nu = \text{Tr}(S_{t,i}S_{t,i}^T|\g),$ and $\beta^2 = \frac{\eta k L^2}{2}||y_0-y^+||^2.$ Let $\mathcal E\in \sigma(\tilde X_{t,0} , \xp, \xpi, \x)$ be an event. Conditioned on $\g$, there exists a coupling of $\y,H_{t,i}$ and $\z$ such that under Assumption \ref{well_conditioned_F},
\[\e[||\z+S_{t,i} - \y||^2|\g] \leq (\is + \ib) \cdot 2\nu + \mathbf 1_{\mathcal E^c} \cdot \frac{11}{2}\nu^2.\] 
\end{corollary}
\begin{proof}
Under the event $\mathcal E$, we couple the Gaussians $\y$ and $\z$ identically (i.e, $\y = \z$). This gives $\e[\|\z+S_{t,i} - \y\|^2|\g] = \e[\|S_{t,i}\|^2|\g] = \nu.$ Under $\mathcal E^c,$ couple them as in the Lemma \ref{new_CLT}.
\end{proof}
\begin{remark}
    Note that $\e(\bern)^2 \leq 1/k,$ so $\nu \leq \eta^2L^2||\xt - \xpi||^2.$ The above Corollary is a slight technical modification of Lemma \ref{new_CLT}. We later choose $\mathcal E$ so that we may neglect terms proportional to $||\nabla f(\xt)||^4,$ arising from our bounds on $\nu^2.$ This is detailed in Lemma \ref{coupling_bound_final}.
\end{remark}
With the above results, we produce an explicit coupling of $X_{t,i+1}$ and $\tilde X_{t,i+1}$ to bound the Wasserstein distance between their distributions. This involves coupling $X_{t,i}$ optimally with $\txtki,$ and bounding movement terms of the form $\e||\tilde X_{t,i}-\tilde X_{t,0}||^p$ and $\e ||\tilde X_{t,i}^+ - \tilde X_{t,0}||^p.$ These moments can be reduced to gradient and Gaussian terms, using the following Lemma. 

\begin{lemma}[Lemma 12, Kandasamy \& Nagaraj~\cite{kandasamy2024poissonmidpointmethodlangevin}]\label{auxiliary} Let $\mtk = \sup_{0 \leq j< k}||\sum_{i=0}^j \sqrt{\frac{2\eta}{k}} \z||$, and $p \in \mathbb N.$ Let $N_t := \sum_{i=0}^{k-1} H_{t,i}.$ Then the following bounds are true. 
        \begin{align*}
        \sup_{0 \leq i \leq k-1} ||\xpi - \tilde X_{t,0}||  & \leq \eta ||\gradf{\tilde X_{t,0}}|| + M_{t,k}. \\
        \sup_{0\leq i \leq k-1} ||\xpi - \xp|| & \leq \eta L N_t \sup_{i \leq k-1}||\xpi - \tilde X_{t,0}||. \\ 
        \e[\mtk^p] & \leq (\eta d)^{p/2}.
        \end{align*}
\end{lemma}

The following Lemma is proven in Section~\ref{sec:recursion_proof}.
\begin{lemma}\label{recursion} Assume $\eta L/k \leq 1,$ and Assumption \ref{well_conditioned_F}. Then there exist absolute constants $c_1,c_2 > 0$ such that 
    \begin{align*}
        \W{\law(X_{t,i+1})}{\law(\tilde X_{t,i+1})} & \leq (1-\frac{\alpha\eta}{2k})\W{\law(\tilde X_{t,i})}{\law(X_{t,i})} + E_{t,i}, \text{ where}
    \end{align*}
    \begin{align*}
        E_{t,i} \lesssim & \Big(\eta^6L^4d+\frac{\eta^4 L^2}{k} + \frac{\eta^5 L^4}{\alpha k}\Big) \e||\gradf{\xt}||^2 \\& + \frac{\eta^4 L^4 d}{\alpha k}
        + \frac{\eta^5L^4d^2}{k} + \exp(c_1 d-(c_2\eta^2 L^2 k)^{-1}) \cdot \frac{\eta^3 L^2 d}{k}.
    \end{align*}
\end{lemma} 
\textbf{Finishing the proof.} Open the recursion in Lemma \ref{recursion}, summing the constant terms as a geometric series.  
\begin{align*}
    \W{X_{t,0}}{\tilde X_{t,0}} & \lesssim \exp(-\alpha\eta t)\W{X_{0,0}}{\tilde X_{0,0}}^2 + (\eta^6L^4kd+ {\eta^4 L^2} + \tfrac{\eta^5 L^4}{\alpha}) \sum_{s=0}^{t-1} \e ||\nabla f(\tilde{X}_{s,0})||^2 \\ & + \frac{\eta^3L^4d}{\alpha^2}+ \frac{\eta^4 L^4 d^2}{\alpha}+\exp(c_1d-(c_2\eta^2 L^2 k)^{-1})\cdot \frac{\eta^2 L^2d}{\alpha}.
\end{align*}
Note that $X_{0,0} = \tilde X_{0,0}$, so $\W{X_{0,0}}{\tilde X_{0,0}} = 0.$ The gradient term $\sum_{t=0}^{N-1} \e ||\nabla f(\xt)||^2$ is bounded in the following Lemma \ref{OLMCgradient}, proven in Section \ref{OLMCgradient_proof}.
\section{Deferred Proofs for Overdamped PLMC}
\subsection{Proof of Proposition \ref{OLMCamgm}}\label{OLMCamgm_proof}
Let $T(x) = x - \frac{\eta}{k} \nabla f(x).$ Under the assumption $\alpha\eta/k < 1,$ it follows from the strong convexity and smoothness of $f$ that $T$ is a contraction with Lipschitz constant $(1-\frac{\alpha\eta}{k}).$ By definition, we have
    \[X_{t,i+1} = T(\x) + \sqrt{\frac{2\eta}{k}}\y, \text{ and } \tilde X_{t,i+1} = T(\xp) + \sqrt{\frac{2\eta}{k}}\zt.\]

    Under the event $\g$, we have:
    \begin{align*}
        \|X_{t,i+1} - \tilde X_{t,i+1}\|^2 & = \|T(x) - T(y)\|^2 + \frac{2\eta}{k} \|\y - \tilde Z_{t,i}\|^2 \\ & + 2 \sqrt{\frac{2\eta}{k}} \langle \y - \tilde Z_{t,i}, T(x) - T(y) \rangle
    \end{align*}
    
    Conditioned on $\g$, $(H_{t,i}-1/k)$ has zero mean, and $Y_{t,i}$, $\z$ are standard Gaussians. This leads to 
    \begin{align*}
        \e[ \|X_{t,i+1} - \tilde X_{t,i+1}\|^2 | \g] & = \|T(x)-T(y)\|^2 - {\frac{2\eta}{k}} \langle \nabla f(y) - \nabla f(y^+), T(x)-T(y) \rangle \nonumber \\&\quad + \frac{2\eta}{k} \e [\|\y - \tilde Z_{t,i}\|^2 | \g] \\
        & \leq (1-\frac{\alpha\eta}{k})^2\|x-y\|^2 + \frac{2\eta L}{k}(1-\frac{\alpha\eta}{k})\|y-y^+\|\cdot \|x-y\| \nonumber \\&\quad + \frac{2\eta}{k} \e [\|\y - \tilde Z_{t,i}\|^2 | \g]. 
    \end{align*}
    
    The second term is bounded using the AM-GM inequality. For any arbitrary $\gamma > 0,$ $$\frac{2\eta L}{k}\|y-y^+\|\cdot \|x-y\| \leq \frac{4\eta^2 L^2}{\gamma}\|y-y^+\|^2 + \frac{\gamma}{k^2}\|x-y\|^2. $$
    In particular, with $\gamma = \alpha\eta k/2,$ 
    \begin{align*}
        &(1-\frac{\alpha\eta}{k})^2\|x-y\|^2 + \frac{2\eta L}{k}(1-\frac{\alpha\eta}{k})\|y-y^+\|\cdot \|x-y\| \\
        & \leq (1-\frac{\alpha\eta}{k})(1-\frac{\alpha\eta}{2k})\|x-y\|^2+(1-\frac{\alpha\eta}{k})\frac{8\eta L^2}{\alpha k}\|y-y^+\|^2 \\
        & \leq (1-\frac{\alpha\eta}{2k})^2 \|x-y\|^2 + \frac{8\eta L^2}{\alpha k}\|y-y^+\|^2.
    \end{align*}
    By definition of $\zt$, \[\zt-\y = \sqrt{\frac{\eta}{2k}}(\gradf{y}-\gradf{y^+}) + \z + S_{t,i}- \y. \]
    Square both sides, noting that $\e[\z+S_{t,i}-\y|\g] = 0,$ and $\|\gradf{y}-\gradf{y^+}\|^2 \leq L^2 \|y-y^+\|^2$ under assumption \ref{well_conditioned_F}. This gives
    \begin{align*}
        \frac{2\eta}{k}\e [\|\y - \tilde Z_{t,i}\|^2 | \g] & = \frac{\eta^2 L^2}{k^2}\|y - y^+\|^2 + \frac{2\eta}{k}\e[\|Z_{t,i}+S_{t,i} - \y\|^2|\g] \\ &
        \leq \frac{\eta L^2}{\alpha k}\|y-y^+\|^2 + \frac{2\eta}{k} \e[\|\z+S_{t,i} - \y\|^2|\g].
    \end{align*}

\subsection{Proof of Lemma~\ref{recursion}}
\label{sec:recursion_proof}
\begin{proof}
Recall the definition
$\g := \{\tilde X_{t,0} = y_0, \xp = y, \xpi = y^+, \x = x\}$. Conditioned on $\g$, we have:
    \[X_{t,i+1} = x - \frac{\eta}{k} \nabla f(x) + \sqrt{\frac{2\eta}{k}}\y\]
    \[\tilde X_{t,i+1} = y + \eta H_{t,i} (\nabla f(y_0) - \nabla f(y^+)) - \frac{\eta}
    {k} \nabla f(y_0) + \sqrt{\frac{2\eta}{k}} \z.\]

Conditioned on $\g$, we couple $(\z, H_{t,i})$ and $\y$ as in Corollary \ref{conditionalCoupling}. This allows us to define $(X_{t,i+1},\tilde X_{t,i+1})$ using the equations above and gives a conditional coupling of $(\y, H_{t,i}, \z, X_{t,i+1},\tilde X_{t,i+1} )$. 

    We produce an unconditional coupling as follows: Couple $X_{t,i}$ and $\txtki$ optimally w.r.t. to $\mathcal W_2$, then sample $\xpi$ and $\xt$ jointly conditioned on $\txtki$. Conditioned on $(\xpi, \xt, X_{t,i}, \txtki)$ (i.e, $\sigma(\tilde X_{t,0} , \xp, \xpi, \x)$), we then sample $(\z, Y_{t,i}, H_{t,i}, X_{t,i+1}, \tilde X_{t,i+1})$ from the conditional coupling described above.  Taking the expectation in Proposition \ref{OLMCamgm}, and using the bounds in Corollary~\ref{conditionalCoupling} we get: 
    \begin{align*}
        \W{X_{t,i+1}}{\tilde X_{t,i+1}}^2 \leq (1-\frac{\alpha\eta}{2k})^2\W{X_{t,i}}{\txtki}^2 + E_{t,i}, 
    \end{align*}
    where $E_{t,i} \lesssim \frac{\eta L^2}{\alpha k}\e||\txtki-\xpi||^2 + \ebk\e[(\is + \ib) \cdot 2\nu + \mathbf 1_{\mathcal E^c} \cdot \frac{11}{2}\nu^2]$ and $\mathcal E\in \sigma(\tilde X_{t,0} , \xp, \xpi, \x)$ is any event. We choose a particular event $\mathcal E$ and bound the latter term in Lemma~\ref{coupling_bound_final}. 
    The former term is bounded below, using items 1 and 2 of Lemma \ref{auxiliary}. 

    \begin{align*}
        \frac{\eta L^2}{\alpha k}\e||\txtki-\xpi||^2 \lesssim \frac{\eta^3 L^4}{\alpha k}\e \left[ N_t^2 \sup_{j\leq k-1}||\xt-{\tilde X}_{t,j}^+||^2 \right]
    \end{align*}
    Note that $N_t$ is independent of $y_0$ and $y^+,$ and $\e[N_t^2] \lesssim 1.$ Along with item 2 of Lemma \ref{auxiliary}, this gives
    \begin{align*}
        \frac{\eta L^2}{\alpha k}\e||\txtki-\xpi||^2 & \lesssim \frac{\eta^5 L^4}{\alpha k}\e||\nabla f(\xt)||^2 + \frac{\eta^3 L^4}{\alpha k}\e[\mtk^2] \\ & 
        \lesssim \frac{\eta^5 L^4}{\alpha k}\e||\nabla f(\xt)||^2 + \frac{\eta^4 L^4 d}{\alpha k}.
    \end{align*}
\end{proof}

\subsection{Proof of Lemma \ref{OLMCgradient}}\label{OLMCgradient_proof}
\begin{proof}
    Since $f$ is smooth, we have (Lemma 3.4, \cite{bubeck2015convexoptimizationalgorithmscomplexity}) \[f(\tilde X_{t+1,0})-f(\xt) \leq \langle \gradf{\xt}, \tilde X_{t+1,0} - \xt \rangle + \frac{L}{2}\|\tilde X_{t+1,0} - \xt\|^2.\]
    By definition, $\tilde X_{t+1,0} - \xt = -\eta \gradf{\xt} + \sum_{i=0}^{k-1}\eta H_{t,i}(\gradf{\xt} - \gradf{\xpi}) + \sum_{i=0}^{k-1}\sqrt\frac{2\eta}{k}\z.$ Since $\e[H_{t,i}] = 1/k$ and $\e[\z] = 0,$
    \begin{align*}
        \e\langle \gradf{\xt}, \tilde X_{t+1,0} - \xt \rangle & \leq -\eta\e\|\gradf{\xt}\|^2 \\& + \sum_{i=0}^{k-1}\frac{\eta}{k}\e\|\gradf{\xt}\|\cdot\|\gradf{\xt}-\gradf{\xpi}\| \\
        & \leq -\frac{\eta}{2}\e\|\gradf{\xt}\|^2 + \sum_{i=0}^{k-1}\frac{\eta}{2k}\e\|\gradf{\xt}-\gradf{\xpi}\|^2 \\
        & \leq -\frac{\eta}{2}\e\|\gradf{\xt}\|^2 + \frac{\eta}{2}\sup_{0 \leq i \leq k-1}\e\|\gradf{\xt}-\gradf{\xpi}\|^2 \\
        & \leq -\frac{\eta}{2}\e\|\gradf{\xt}\|^2 + \eta^3 L^2 \|\gradf{\xt}\|^2 + \eta^2 L^2 d.
    \end{align*}
    Where in the second and final steps we used $ab \leq \frac{a^2+b^2}{2}$ and Lemma \ref{auxiliary} respectively. Now we use $\|a+b\|^2 \leq 2(\|a\|^2 + \|b\|^2)$ and $\e\|\sum_{i=0}^{k-1}\sqrt\frac{2\eta}{k}\z\|^2 = 2\eta d$ to get  
    \begin{align*}
        \frac{L}{2}\|\tilde X_{t+1,0} - \xt\|^2 \leq \eta^2L \|\gradf{\xt}\|^2 + \eta^2 L \|\sum_{i=0}^{k-1}H_{t,i}(\gradf{\xt}-\gradf{\xpi})\|^2 + 2\eta L d.
    \end{align*}
    Let $N_t = \sum_{i=0}^{k-1}H_{t,i}$. Note that $\e[N_t^2] \leq 2,$ and $N_t$ is independent of $\xt.$ Triangle inequality and \ref{auxiliary} give
    \begin{align*}
        \eta^2 L \e\|\sum_{i=0}^{k-1}H_{t,i}(\gradf{\xt}-\gradf{\xpi})\|^2 & \leq \ \eta^2 L \e[N_t \sup_{0 \leq i \leq k-1}\e\|\gradf{\xt}-\gradf{\xpi}\|]^2 \\
        & \leq 4\eta^4 L^3 \e\|\gradf{\xt}\|^2 + 4\eta^3 L^3 d.
    \end{align*}
    Under our assumption $\eta L \leq 1/8,$ the terms $\eta^3 L^2 \|\gradf{\xt}\|^2, \eta^4 L^3 \e\|\gradf{\xt}\|^2,\eta^2 L^2 d$ and $\eta^3 L^3 d$ are negligible in order. Collecting the dominant terms, we get 
    \begin{align*}
        \eta \e\|\gradf{\xt}\|^2 \lesssim [f(\xt) - f(\tilde X_{t+1,0})] + \eta Ld. 
    \end{align*}
    This telescopes, leading to the result.
\end{proof}

\subsection{Proof of Corollary~\ref{OLMCcomplexity}}
\label{sec:OLMCcomplexityproof}
\begin{proof}
    By triangle inequality on $\mathcal W_2,$
    \begin{align*}
        \W{\law(\tilde X_{N,0})}{\pi} \lesssim \W{\law(\tilde X_{N,0})}{\law(X_{N,0})} + \W{\law(X_{N,0})}{\pi}.
    \end{align*}
    We show under the conditions of our Corollary that each of these terms is $\mathcal O(\epsilon^2d/\alpha).$
    To deal with the second term, recall the following Theorem for the convergence of Langevin Monte-Carlo. 
    \begin{theorem}[Corollary 10, Durmus et al.~\cite{durmus2018analysislangevinmontecarlo}]
        Suppose Assumption \ref{well_conditioned_F} is true. Let $X_n$ denote the iterates of Langevin Monte-Carlo with step-size $\gamma_\epsilon.$
        Then, with 
        \[
            \gamma_\epsilon = \frac{\epsilon^2}{4L}, \;\;\;\; n_\epsilon \geq \lceil \log(\frac{2\W{X_0}{\pi}\alpha}{\epsilon^2 d})\gamma_\epsilon^{-1}\alpha^{-1}\rceil
        \]
        we have $\W{X_{n_\epsilon}}{\pi} \leq \frac{\epsilon^2 d}{\alpha}.$
    \end{theorem}
    By our choice of $k$, we have $\ebk \lesssim \frac{\epsilon^2}{L}.$ Note that the above Theorem goes through with an inequality $\eta \leq \frac{\epsilon^2}{4L}$, so we have $\W{X_{N,0}}{\pi} \leq \frac{\epsilon^2d}{\alpha}$ for $N = \log(\frac{2\W{X_{0,0}}{\pi}\alpha}{\epsilon^2 d})(\alpha\eta)^{-1}.$ 
    Let $L_1 = 2\max(C_f, \log(\frac{2\W{X_{0,0}}{\pi}\alpha}{\epsilon^2 d})).$ Now apply Theorem \ref{OLMCmain} with 
    \[\eta \asymp \min\Big(\frac{\epsilon^{2/3}}{L_1^{1/3}L},\frac{\epsilon^{1/2}}{\kappa^{1/4}L_1^{1/4}L},\frac{\epsilon^{2/3}}{d^{1/6}L_1^{1/6}L}, \frac{\epsilon^{2/3}}{\kappa^{1/3}L},\frac{\epsilon^{1/2}}{d^{1/4}L},(\frac{c_2\epsilon^2}{c_1d - \log \epsilon^2})^{1/3} \cdot \frac{1}{L}\Big)\]
    and $N$ as above, to see $\W{\law(\tilde X_{N,0})}{\law(X_{N,0})} \lesssim \frac{\epsilon^2d}{\alpha}.$ 
\end{proof}

\section{Technical Results for OLMC}
\begin{lemma}\label{coupling_bound_final}
    Let $\beta$ and $\nu$ be defined as in Lemma \ref{conditionalCoupling}. Define the event $\mathcal E \in \sigma(\tilde X_{t,0} , \xp, \xpi, \x)$ by $\mathcal E = \{\frac{\eta^4L^2}{k}\|\gradf{\xt}\|^2 < \frac{\eta^7 L^4}{k}\|\gradf{\xt}\|^4 \}.$ Then
    \begin{align*}
        \ebk \e[(\is + \ib) \cdot \nu + \mathbf 1_{\mathcal E^c} \cdot \nu^2] & \lesssim (\eta^6L^4d+\frac{\eta^4 L^2}{k}) \e\|\gradf{\xt}\|^2 + \frac{\eta^5L^4 d^2}{k} \\ &  + \exp(c_1d-(c_2 \eta^2 L^2 k)^{-1}) \cdot \frac{\eta^3 L^2 d}{k}
    \end{align*}
    where the expectation is taken over the distribution defined in the proof of \ref{recursion}.
\end{lemma}
\begin{proof}
    Since $H_{t,i}$ is a Bernoulli random variable with parameter $1/k,$ we have $\e[(H_{t,i}-1/k)^2] \leq 1/k.$ This gives us an upper bound on $\nu$, since $\nu  = \e[\frac{\eta k}{2}(H_{t,i}-1/k)^2 \|\nabla f(y_0)-\nabla f(y^+)\|^2] \leq \frac{\eta L^2}{2}\|y_0-y^+\|^2$ under Assumption \ref{well_conditioned_F}. This gives 
    \[\ebk[(\is + \ib) \cdot \nu + \mathbf 1_{\mathcal E^c} \cdot \nu^2] \lesssim (\is + \ib)\cdot \frac{\eta^2 L^2}{k}\|y_0-y^+\|^2 + \mathbf 1_{\mathcal E^c} \cdot \frac{\eta^3 L^4}{k}\|y_0- y^+\|^4\]
    Now we apply item 1 of Lemma \ref{auxiliary} to obtain the following. 
    \begin{align*}
        \frac{\eta^2 L^2}{k}\|y_0-y^+\|^2 & \lesssim \frac{\eta^4L^2}{k}\|\gradf{\xt}\|^2 + \frac{\eta^2 L^2 }{k}\mtk^2. \\ 
        \frac{\eta^3 L^4}{k}\|y_0 -y^+\|^4 & \lesssim \frac{\eta^7 L^4}{k}\|\gradf{\xt}\|^4 + \frac{\eta^3 L^4}{k}\mtk^4. 
    \end{align*}
    Using $(\ib+\is) \lesssim 1$ and $\mathbf 1_{\mathcal E^c} \frac{\eta^7 L^4}{k}\|\nabla f(\xt)\|^4 \leq \frac{\eta^4 L^2}{k}\|\nabla f(\xt)\|^2,$ we obtain  
    \begin{align*}
        \ebk[(\is + \ib) \cdot \nu + \mathbf 1_{\mathcal E^c} \cdot \nu^2] & \lesssim \frac{\eta^4 L^2}{k}\|\nabla f(\xt)\|^2 + \frac{\eta^3 L^4}{k}\mtk^4 \\ & + (\is+\ib)\frac{\eta^2L^2}{k}\mtk^2.
    \end{align*}
    The expectations of the second term and final terms are bounded in Lemmas \ref{auxiliary} and \ref{noise_bound} respectively. 
\end{proof}

\begin{lemma}\label{noise_bound}
    Let $\beta$ and $\mathcal E$ be as in Lemma \ref{coupling_bound_final}. There exists an absolute constants $c_1$ and $c_2$ such that 
    \begin{align*}
        \e[(\ib + \is) \cdot \frac{\eta^2 L^2}{k}\mtk^2] \lesssim \eta^6 L^4 d \e\|\gradf{\xt}\|^2 + \exp(c_1{d} - (c_2 \eta^2 L^2 k)^{-1})\frac{\eta^3L^2d}{k}
    \end{align*}
\end{lemma}
\begin{proof}
    Note that $\mathcal E$ is independent of $\mtk,$ and by its definition we have $1_\mathcal E \leq \eta^3 L^2 \|\nabla f(\xt)\|^2.$ As a result, 
    
    \[\e[1_\mathcal E\cdot \frac{\eta^2 L^2}{k}\mtk^2] \leq \eta^3 L^2 \e\|\gradf{\xt}\|^2\cdot \e[\frac{\eta^2 L^2}{k}\mtk^2].\]
    Recall the definition of $\beta.$
    \begin{align*}
        \beta &\leq \sqrt{\eta k}L\|\xt - \xpi\| \\
            & = \sqrt{\eta k}L\biggr\|\frac{\eta i}{k} \gradf{\xt} + \sqrt{\frac{2\eta}{k}}\sum_{j=0}^i Z_{t,j}\biggr\|.
        \intertext{Applying triangle inequality and union bound, we get }
        \mathbf 1_{\sqrt 5 \beta > 1} & \leq \mathbf 1\{\sqrt 5\eta^{3/2}k^{1/2}L\|\gradf{\xt}\| > 1\} + \mathbf 1\{\sqrt 10 \eta L \|\sum_{j=0}^i Z_{t,j}\| > 1\}.
        \intertext{Note that $\xt$ is independent of $\mtk$. To handle the second term below, apply Cauchy Schwarz and a Gaussian concentration inequality. }
        \e[\ib \cdot \frac{\eta^2 L^2}{k} \mtk^2] & \leq \p[\sqrt 5\eta^{3/2}k^{1/2}L\|\gradf{\xt}\| > 1] \cdot \frac{\eta^2 L^2}{k} \e[\mtk^2] \\ & + \p\big[\sqrt 10 \eta L \|\sum_{j=0}^i Z_{t,j}\| > 1\big]^{1/2} \cdot \frac{\eta^2 L^2}{k} \e[\mtk^4]^{1/2} \\
        & \lesssim \eta^3kL^2 \e[\|\gradf{\xt}\|^2] \cdot \frac{\eta^2 L^2}{k}\e[\mtk^2] \\ & + \exp(c_1 d - (c_2\eta^2L^2 k)^{-1}) \cdot \frac{\eta^2 L^2}{k}\e[\mtk^4]^{1/2}.
    \end{align*}
    Where $c_1, c_2> 0$ are absolute constants. Applying item 2 of Lemma \ref{auxiliary} completes the proof. 
\end{proof}
\section{Proof for Underdamped PLMC}\label{ulmc_main_proof}
\subsection{Basis change for contractivity}
Recall from Section \ref{plmc_definitions} the definitions of $\tutki, \tvtki.$ We make the following coordinate change for the iterates of underdamped LMC/PLMC.  
\[\begin{bmatrix}
    x\\y
\end{bmatrix} \to \mathcal{M} \begin{bmatrix}
    x\\y
\end{bmatrix}, \text{   where   } \mathcal M = \begin{bmatrix}
    \vI_d & 0 \\ \vI_d & \frac{2}{\gamma}\vI_d
\end{bmatrix}.\]
We denote $W_{t,i} = U_{t,i} + \frac{2}{\gamma}V_{t,i},$ and $\tilde W_{t,i} = \tilde U_{t,i} + \frac{2}{\gamma}\tilde V_{t,i}.$ Similarly, $\tilde W_{t,i}^+ = \tilde U_{t,i}^+ + \frac{2}{\gamma}\tilde V_{t,i}^+,$ and 
\[\tilde X_{t,i} = \begin{bmatrix}
    \tilde U_{t,i} \\ \tilde W_{t,i}
\end{bmatrix}, \tilde X_{t,i}^+ = \begin{bmatrix}
    \tilde U_{t,i}^+ \\ \tilde W_{t,i}^+
\end{bmatrix} \text{, and } X_{t,i} = \begin{bmatrix}
    U_{t,i} \\ W_{t,i}
\end{bmatrix}.\]

The transformed iterates $\tutki,$ $\twtki$ satisfy the following recursion. 
\begin{align*}
\begin{bmatrix}
    \tilde U_{t,i+1} \\ \tilde W_{t,i+1}
\end{bmatrix} = A_\mathcal M\Big(\frac{\eta}{k}\Big) \begin{bmatrix}
    \tilde U_{t,i} \\ \tilde W_{t,i}
\end{bmatrix} & - G_\mathcal M\Big(\frac{\eta}{k}\Big)\begin{bmatrix}
    \nabla f(\tilde U_{t,0}) \\ 0
\end{bmatrix} + \Gamma_\mathcal M\Big(\frac{\eta}{k}\Big) Z_{t,i} \\ & + k H_{t,i}\cdot G_\mathcal M\Big(\frac{\eta}{k}\Big)
\begin{bmatrix}
    \nabla f(\tilde U_{t,0}) - \nabla f(\tilde U_{t,i}^+)\\ 0
\end{bmatrix}
\end{align*}
The matrices $A_{\mathcal M}, G_{\mathcal M}$ and $\Gamma_{\mathcal M}$ account for the change of basis. It can be verified that $\am = \m A\m^{-1},$ and $\gm = \m G.$ Moreover, $\Gm = \m \Gamma,$ and these are explicated below. 
\[A_{\mathcal M}(h) = \begin{bmatrix}
    \frac{1}{2}(1+\exp(-\gamma h))\vI_d & \frac{1}{2}(1-\exp(-\gamma h))\vI_d \\
    \frac{1}{2}(1-\exp(-\gamma h))\vI_d & \frac{1}{2}(1+\exp(-\gamma h))\vI_d
\end{bmatrix}, \text{   } G_{\mathcal M}(h) = \begin{bmatrix}
    \frac{\gamma h-(1-\exp(-\gamma h))}{\gamma^2}\vI_d & 0 \\
    \frac{\gamma h + (1-\exp(-\gamma h))}{\gamma^2}\vI_d & 0
\end{bmatrix}.\]
 \[\Gamma_{\mathcal M}^2(h) = \begin{bmatrix}
    \frac{4(1-\exp(-\gamma h) - (1-\exp(2\gamma h)) + 2\gamma h}{\gamma^2}\vI_d & \frac{2\gamma h - (1-\exp(2\gamma h))}{\gamma^2}\vI_d \\
    \frac{2\gamma h - (1-\exp(2\gamma h))}{\gamma^2}\vI_d & \frac{4(1-\exp(-\gamma h) + (1-\exp(2\gamma h)) + 2\gamma h}{\gamma^2}\vI_d
\end{bmatrix}\]

In order to interpret this as ULMC with perturbed Gaussian noise, we write
\[\begin{bmatrix}
    \tilde U_{t,i+1} \\ \tilde W_{t,i+1}
\end{bmatrix} = A_{\mathcal M}\Big(\ebk\Big)\begin{bmatrix}
    \tutki \\ \twtki
\end{bmatrix}+ G_{\mathcal M}\Big(\ebk\Big) \begin{bmatrix}
    -\nabla f(\tutki) \\ 0
\end{bmatrix} + \Gamma_{\mathcal M}\Big(\ebk\Big) \zt.\] 
The perturbed Gaussian $\zt$ can be expressed as $\zt = \z + B_{t,i} + S_{t,i},$ where
\[B_{t,i} = \Gamma_{\mathcal M}^{-1}\Big(\ebk\Big)G_{\mathcal M}\Big(\ebk\Big)\begin{bmatrix}
    \nabla f(\tutki) - \nabla f(\tutki^+) \\ 0
\end{bmatrix}\]
\[S_{t,i} = k(H_{t,i} - 1/k) \Gamma_{\mathcal M}\Big(\ebk\Big)^{-1} G_{\mathcal M}\Big(\ebk\Big) \begin{bmatrix}
    \nabla f(\utkt) - \nabla f(\tutki^+) \\ 0
\end{bmatrix}.\]
Here $B_{t,i},S_{t,i}$ are called the bias and variance terms respectively. 

The midpoints are given by 
\[\begin{bmatrix}
    \tilde U_{t,i}^+ \\ \tilde W_{t,i}^+
\end{bmatrix} = A_\mathcal M\Big(\frac{\eta i}{k}\Big)\begin{bmatrix}
    \tilde U_{t,0} \\ \tilde W_{t,0}
\end{bmatrix} - G_\mathcal M\Big(\frac{\eta i}{k}\Big)\begin{bmatrix}
    \nabla f(U_{t,0}) \\ 0
\end{bmatrix} + \sum_{j=0}^{i-1}A_\mathcal M\Big(\frac{\eta(k-1-i)}{k}\Big)G_\mathcal M\Big(\frac{\eta i}{k}\Big) Z_{t,i}.\] 

The iterates of underdamped LMC satisfy 
\[\begin{bmatrix}
    U_{t,i+1} \\ W_{t,i+1}
\end{bmatrix} = A_\mathcal M\Big(\frac{\eta}{k}\Big) \begin{bmatrix}
    U_{t,i} \\ W_{t,i}
\end{bmatrix} - G_\mathcal M\Big(\frac{\eta}{k}\Big)\begin{bmatrix}
    \nabla f(U_{t,i}) \\ 0
\end{bmatrix} + \Gamma_\mathcal M\Big(\frac{\eta}{k}\Big)Y_{t,i},\]
\[\begin{bmatrix}
    U_{t+1,0} \\ W_{t+1,0}
\end{bmatrix} = \begin{bmatrix}
    U_{t,k} \\ W_{t,k}
\end{bmatrix}.\]
Here $\y$ and $\z$ are i.i.d. standard Gaussians, $H_{t,i}$ are Bernoulli with parameter $1/k$, and all random variables above live on the same probability space with a coupling yet to be specified. 

\subsection{Proof overview}
Our proof follows the same method as in the overdamped case. As before, 
We condition on the previous iterates -- with the following event: 
    \[\mathcal G = \Big\{ \xt = y_0 = \begin{bmatrix}
        u_0\\w_0
    \end{bmatrix}, \txtki = y = \begin{bmatrix}
        \tilde u \\ \tilde w
    \end{bmatrix}, \xpi = y^+ = \begin{bmatrix}
        u^+ \\ w^+
    \end{bmatrix}, X_{t,i} = x = \begin{bmatrix}
        u\\w
    \end{bmatrix} \Big\},\]

where $y_0,y,y^+$ and $x$  arbitrary points in $\mathbb R^{2d}$. For any valid coupling of $X_{t,i+1}$ and $\tilde X_{t,i+1},$ the following holds. 
\begin{proposition}\label{ULMCamgm} Assume $\eta/k \lesssim \frac{1}{\kappa\sqrt{L}},$ and $\frac{\gamma\eta}{k} < c_0$ for sufficiently small $c_0>0.$ Then with $\gamma = c_\gamma\sqrt L$ for some $c_\gamma \geq 2,$ the following holds. 
    \begin{align*}
        \e[\|X_{t,i+1} - \tilde X_{t,i+1}\|^2 | \g] & \leq (1-\Omega(\frac{\alpha\eta }{\gamma k}))\|x-y\|^2  \\ & + O\Big[ \frac{\eta L^{2}}{\alpha\gamma k}\|u^+ - \tilde u\|^2 + \frac{\eta}{\gamma k} \e[\|\z+S_{t,i}-\y\|^2|\g] \Big].
    \end{align*}
\end{proposition}
The above Proposition is proved in Section~\ref{sec:ULMCamgmproof}. The first term arises from the contractivity of the ULMC update rule, while the second term comes from the bias. Having conditioned on $\g$, we use $\ref{new_CLT}$ to bound the final term $\e[\|\z+S_{t,i}-\y\|^2 \g]$. We refer to Section~\ref{sec:prop_ULMC_coupling_proof} for the proof of the following proposition.
\begin{proposition}\label{ULMC_coupling}
    Let $p \geq 0$ be an integer. Conditioned on $\g$, there exists a coupling of $\z, H_{t,i}$ and $\y$ such that
    \[\tfrac{\eta}{\gamma k} \e[\|\z+S_{t,i}-\y\|^2|\g] \lesssim \tfrac{\eta^3L^4}{\gamma^3 k}\|u_0-u^+\|^4 + \tfrac{5^p\eta^{p+2} k^{p-1} L^{2p+2}}{\gamma^{p+2}}\|u_0-u^+\|^{2p+2}. \]
\end{proposition}
\begin{remark}
    The presence of $p$ is due to the manner in which handle the low probability event $\{5\beta^2>1\},$ appearing in Lemma \ref{new_CLT}. We use $\ib \leq 5^p\beta^{2p},$ with an appropriate bound on $\beta^{2p}$. Each choice of $p$ leads to a different error bound, so we write this in generality. 
\end{remark} 
With the above results, we produce an explicit coupling of $X_{t,i+1}$ and $\tilde X_{t,i+1}$ to bound the Wasserstein distance between their distributions. This is done by coupling $X_{t,i}$ optimally with $\tilde X_{t,i},$ then bounding the moments $\e||u^+-\tilde u||^2$ and $\e||u_0-\tilde u||^p.$ These moments contain gradient, momentum, and Gaussian terms; and are handled via the following Lemma. 
\begin{lemma}[Lemma 21, Kandasamy \& Nagaraj~\cite{kandasamy2024poissonmidpointmethodlangevin}]\label{ULMCauxiliary} Let $\Pi$ denote projection onto the position axis: $\Pi[u,v]^T = [u,0]^T.$ Let $\mtk = \sup_{0\leq i < k}\|\sum_{j=0}^i A(\frac{\eta(i-j)}{k})\Gamma(\ebk) Z_{t,j}\|_\Pi.$ Then the following inequalities are true. 
    \begin{align*}
        \|\tutki^+ - \utkt\| & \lesssim  \eta \|\vtkt\| + \eta^2 \|\nabla f(\utkt)\| + \mtk \\
        \e[\mtk^p] & \lesssim \exp(\frac{p\gamma\eta}{2})\gamma^{p/2}\eta^{3p/2}(d+\log k)^{p/2}.
    \end{align*}
\end{lemma}
The proof of the following Lemma is given in Section~\ref{sec:ULMC_coupling_proof}
\begin{lemma}\label{lem:ULMC_coupling}
Assume $\eta/k \lesssim \frac{1}{\kappa\sqrt{L}},$ and $\gamma\eta < c_0$ for sufficiently small $c_0 > 0.$ With $\gamma = c\sqrt L$ for some $c \geq 2$, the following is true. 
    \begin{align*}
        \mathcal W_2^2&(\law(X_{t,i+1}),\law({\tilde X_{t,i+1}})) \leq (1-\Omega(\tfrac{\alpha\eta}{\gamma k}))\W{X_{t,i}}{\txtki} +  \mathcal O\Big[\tfrac{\eta^7 L^4}{\alpha\gamma k}\e\|\vtkt\|^2 \\ & + \tfrac{\eta^9 L^{4}}{\alpha\gamma k}\e\|\nabla f(\utkt)\|^2 + \tfrac{\eta^8 L^4}{\alpha k}(d+\log k) + \tfrac{\eta^7 L^4}{\gamma^3 k}\e\|\vtkt\|^4 \\ & + \tfrac{\eta^{11} L^4}{\gamma^3 k}\e\|\nabla f(\utkt)\|^4 + \tfrac{\eta^9 L^4}{\gamma k}(d+\log k)^2 + \lambda_p[\tfrac{\eta^{3p+4}k^{p-1}L^{2p+2}}{\gamma^{p+2}}\e\|\vtkt\|^{2p+2} \\ & +\tfrac{\eta^{5p+6}k^{p-1}L^{2p+2}}{\gamma^{p+2}}\e\|\nabla f(\utkt)\|^{2p+2} + \tfrac{\eta^{4p+5}k^{p-1}L^{2p+2}}{\gamma}(d+\log k)^{p+1}\Big], 
    \end{align*} 
    Where $\lambda_p$ is a constant depending only on $p.$
\end{lemma}

\subsection{Finishing the proof}\label{ULMCmain_complete}
    Open up the recursion, summing up the constant terms as a geometric series. This gives 
    \begin{align*}
        W_2^2({\law(\xt)},&{\law(X_{t,0})}) \lesssim \exp\Big(\Omega(-\tfrac{\alpha \eta t}{\sqrt L})\Big)\W{\law(X_{0,0})}{\law(\tilde X_{0,0})} \\ & + \sum_{s=0}^t \Big[\tfrac{\eta^7 L^4}{\alpha\gamma}\e\|\vtkt\|^2 + \tfrac{\eta^9 L^{4}}{\alpha\gamma}\e\|\nabla f(\utkt)\|^2 \Big] + \tfrac{\eta^7 L^4 \gamma}{\alpha^2}(d+\log k) \\&
        + \sum_{s=0}^t\Big[\tfrac{\eta^7 L^4}{\gamma^3}\e\|\vtkt\|^4 + \tfrac{\eta^{11} L^4}{\gamma^3}\e\|\nabla f(\utkt)\|^4 \Big] + \tfrac{\eta^8 L^4}{\alpha}(d+\log k)^2 \\& + \sum_{s=0}^t \lambda_p \Big[\tfrac{\eta^{3p+4}k^{p}L^{2p+2}}{\gamma^{p+2}}\e\|\vtkt\|^{2p+2} +\tfrac{\eta^{5p+6}k^{p}L^{2p+2}}{\gamma^{p+2}}\e\|\nabla f(\utkt)\|^{2p+2}\Big] \\& + \lambda_p\tfrac{\eta^{4p+4}k^{p}L^{2p+2}}{\alpha}(d+\log k)^{p+1}
    \end{align*}
Note that $X_{0,0} = \tilde X_{0,0}$ by definition, so the first term is zero. The moments $\sum_{s=0}^t \e||\vtkt||^{2p}$ and $\sum_{s=0}^t \e||\nabla f(\utkt)||^{2p}$ are bounded the following Lemma. 
\begin{theorem}[Theorem 4, Kandasamy \& Nagaraj~\cite{kandasamy2024poissonmidpointmethodlangevin}]\label{ULMCgradients}
Fix $p \geq 1,$ and let $\mathcal S_{2p}(\nabla f) = \sum_{t=0}^T \e\|\nabla f(\utkt)\|^{2p}$. Let $\mathcal S_{2p}(V) = \sum_{t=0}^T \e\|\vtkt\|^{2p},$ and $\Psi_t = \tilde U_{t,0} + \frac{1}{\gamma}\tilde V_{t,0}.$ There exist constants $C_p,c_p,\bar{c}_p > 0$ such that whenever:
$\gamma \geq C_p \sqrt{L}$, $\alpha \gamma < c_p$, $\frac{\eta^{3p-1}T^{p-1}L^{2p}}{\gamma^{p+1}} < \bar{c}_p $, the following results hold:

\begin{align*}\mathcal{S}_{2p}(\nabla f) &\leq C_p \frac{\gamma^{2p-1}}{\eta}\left[\e \|\tilde{V}_{0,0}\|^{2p} + \e |(f(\Psi_0)-f(\Psi_T))^{+}|^p +1\right] + \\&\quad  C_p T\left[\tfrac{\gamma^{4p}}{L^p} + (\gamma\eta T)^{p-1}\gamma^{2p}\right](d+\log k)^p\end{align*}

\begin{align*}\mathcal{S}_{2p}(V) &\leq C_p \frac{1}{\gamma\eta}\left[\e \|\tilde{V}_{0,0}\|^{2p} + \e |(f(\Psi_0)-f(\Psi_T))^{+}|^p +1\right] \nonumber \\&\quad + C_p T\left[\tfrac{\gamma^{2p}}{L^p} + (\gamma\eta T)^{p-1}\right](d+\log k)^p\end{align*}
\end{theorem}
\begin{remark}\label{suboptimality_ulmc}
    We believe these bounds are suboptimal. When $V$ is a standard Gaussian random vector, we have $\e||V||^{2p} = d^p.$ Similarly, when $f$ is L-smooth, it can be shown that \[\int_{\sR^d}||\nabla f(x)||^{2p} d\pi(x) \leq \prod_{n=1}^p(2n-1) \cdot (Ld)^p.\]
    This is Lemma \ref{smooth_gradient_bounds}, and is a generalization of Lemma 11 from Vempala \& Wibisono~\cite{vempala2022rapidconvergenceunadjustedlangevin}.
    We thus believe the dominant term in both bounds should be $\mathcal O(Td^p),$ whereas what we have is $\mathcal O(\eta^{p-1}T^p d^p).$ When $T\asymp 1/\alpha\eta,$ this is suboptimal in $\kappa$ dependence. 
\end{remark}
We substitute the bounds from \ref{ULMCgradients}, ignoring lower order terms via the assumption $\gamma\eta <c_0.$ 
\begin{align*}\label{ulmc_complete}
    W_2^2(&{\law(\xt)},{\law(X_{t,0})}) \lesssim \tfrac{\eta^6 L^{4}}{\alpha\gamma^2}\left[\e \|\tilde{V}_0\|^{2} + \e |(f(\Psi_0)-f(\Psi_T))^{+} +1\right] \\ & + \tfrac{\eta^7 L^{4}}{\alpha\gamma}t\left[\tfrac{\gamma^{2}}{L} + 1\right](d+\log k) + \tfrac{\eta^6 L^4}{\gamma^4}\left[\e \|\tilde{V}_0\|^{4} + \e |(f(\Psi_0)-f(\Psi_T))^{+}|^2 +1\right] \\ & 
    + \tfrac{\eta^7 L^4}{\gamma^3}t\left[\tfrac{\gamma^{4}}{L^2} + (\gamma\eta t)\right](d+\log k)^2 \\ &
    + \lambda_p\tfrac{\eta^{3p+3}k^{p}L^{2p+2}}{\gamma^{p+3}}\left[\e \|\tilde{V}_0\|^{2p+2} + \e |(f(\Psi_0)-f(\Psi_T))^{+}|^{p+1} +1\right] \\ &
    + \lambda_p \tfrac{\eta^{3p+4}k^{p}L^{2p+2}}{\gamma^{p+2}}t\left[\tfrac{\gamma^{2p+2}}{L^{p+1}} + (\gamma\eta t)^{p}\right](d+\log k)^{p+1} \\ & +\tfrac{\eta^7 L^4 \gamma}{\alpha^2}(d+\log k)+\tfrac{\eta^8 L^4}{\alpha}(d+\log k)^2+\lambda_p\tfrac{\eta^{4p+4}k^{p}L^{2p+2}}{\alpha}(d+\log k)^{p+1}.
\end{align*}
\section{Deferred Proofs for ULMC}
\subsection{Proof of Proposition~\ref{ULMC_coupling}}
\label{sec:prop_ULMC_coupling_proof}
\begin{proof}
Let $\beta = \sqrt{\frac{\eta k}{\gamma}}L\|u_0-u^+\|.$ By Proposition \ref{spectralNormBounds} we have $\|G_\mathcal M(\ebk)^T \Gamma_\mathcal M(\ebk)^{-2}G_\mathcal M(\ebk)\| \leq \frac{\eta}{\gamma k},$ and we know $H_{t,i} \leq 1$ since it is a Bernoulli. It follows that $\|S_{t,i}\|^2 \leq \beta^2$.  Now let $\nu = \textrm{Tr}(\e[S_{t,i}S_{t,i}^T|\g]).$ Since $\e[(H_{t,i}-1/k)^2] \leq 1/k,$ it follows that $\nu \leq \frac{\eta L^2}{\gamma}\|u_0-u^+\|^2.$ Applying Lemma \ref{new_CLT} gives 
    \begin{align*}
        \e[\|\z+S_{t,i}-\y\|^2|\g] & \lesssim \frac{\eta^2 L^4}{\gamma^2}\|u_0-u^+\|^4 + \ib \cdot \frac{\eta L^2}{\gamma}\|u_0-u^+\|^2 \\
        & \lesssim \frac{\eta^2 L^4}{\gamma^2}\e\|u_0-u^+\|^4 + \frac{5^p\eta^{p+1} k^{p} L^{2p+2}}{\gamma^{p+1}}\e\|u_0-u^+\|^{2p+2}. 
    \end{align*}
    In the last line, we have used $\ib \leq (5\beta^2)^{p} = \frac{5^p\eta^p k^p L^{2p}}{\gamma^p}\|u_0 -u^+\|^{2p}.$ Multiplying this inequality by $\frac{\eta}{\gamma k}$ finishes the proof. 
\end{proof}

\subsection{Proof of Lemma~\ref{lem:ULMC_coupling}}
\label{sec:ULMC_coupling_proof}
We will use the following bounds in the proof. 
\begin{lemma}[Lemmas 18/19, Kandasamy \& Nagaraj~\cite{kandasamy2024poissonmidpointmethodlangevin}]\label{ULMCauxiliary_2}
    Let $\Pi:\mathbb R^{2d} \to \mathbb R^{2d}$ denote projection onto the first $d$ coordinates. Let $G(h)$ and $A(h)$ be as defined in the update rule for underdamped Langevin Monte-Carlo in Section \ref{plmc_definitions}. Let $\tutki^+, \tvtki^+$ and $\tutki, \tvtki$ denote the midpoints and iterates respectively of Poisson-ULMC, as defined in Section \ref{plmc_definitions}. Let $\|\cdot\|$ denote the operator norm of a matrix, and $\|\cdot\|_{\mathbb R^n}$ denote the Euclidean norm in dimension $n$. Then the following inequalities are true.   

    \[\Bigg\|\Pi A_\mathcal M\Big({\frac{j\eta}{k}}\Big) G_\mathcal M\Big(\ebk\Big)\Bigg\| \lesssim \frac{\eta^2}{k},\]  
    \[\Bigg\|\begin{bmatrix}
        \tutki^+-\tutki\\ \tvtki^+-\tvtki
    \end{bmatrix}\Bigg\|_{\mathbb R^{2d}} \leq \sum_{j=0}^{i-1} k H_{t,j} \Bigg\|A\Big(\frac{(i-j-1)\eta}{k}\Big) G\Big(\ebk\Big)\begin{bmatrix}
        \nabla f(\utkt)-\nabla f(\tutki^+) \\ 0
    \end{bmatrix}\Bigg\|_{\mathbb R^{2d}}.\]

\end{lemma}

\begin{proof}
Recall the definition of $\mathcal G.$
\[\mathcal G = \Big\{ \xt = y_0 = \begin{bmatrix}
        u_0\\w_0
    \end{bmatrix}, \txtki = y = \begin{bmatrix}
        \tilde u \\ \tilde w
    \end{bmatrix}, \xpi = y^+ = \begin{bmatrix}
        u^+ \\ w^+
    \end{bmatrix}, X_{t,i} = x = \begin{bmatrix}
        u\\w
    \end{bmatrix} \Big\}.\]
By definition, conditioned on $\g$, we have 
\[X_{t,i+1} = A_\mathcal M\Big(\ebk\Big)x + \gm\Big(\ebk\Big)\begin{bmatrix}
    -\nabla f(u) \\ 0
\end{bmatrix} + \Gm \y,\]
\begin{align*}
    \tilde X_{t,i+1} = \am \Big(\ebk\Big)y & + \gm\Big(\ebk\Big) \begin{bmatrix}
        -\nabla f(\tilde u)\\0
    \end{bmatrix} + \Gamma_\mathcal M\Big(\frac{\eta}{k}\Big) Z_{t,i} \\ & + k H_{t,i}\cdot G_\mathcal M\Big(\frac{\eta}{k}\Big)
    \begin{bmatrix}
    \nabla f(u_0) - \nabla f(u^+)\\ 0
    \end{bmatrix}
\end{align*}
Conditioned on $\g,$ we couple $\z, H_{t,i}$ and $\y$ as in Lemma \ref{ULMC_coupling}. This allows us to define $(X_{t,i+1},\tilde X_{t,i+1})$ using the equations above and gives a conditional coupling of $(\z, H_{t,i}, \y, X_{t,i+1},\tilde X_{t,i+1})$ given $\g$. \\ 
We produce an unconditional coupling as follows. Couple $X_{t,i}$ and $\txtki$ optimally w.r.t. $W_2$, then sample $\xpi$ and $\xt$ jointly conditioned on $\txtki.$ Conditioned on $(X_{t,i},\txtki,\xpi,\xt)$ we then sample $(\z,\y,H_{t,i},X_{t,i+1}, \tilde X_{t,i+1})$ from the conditional coupling described above. We now take the expectation in Proposition \ref{ULMCamgm}, after substituting the bound in Proposition \ref{ULMC_coupling}. This gives 
\begin{align*}
    \mathcal W_2^2({X_{t,i+1}},&{\tilde X_{t,i+1}}) \leq (1-\Omega(\frac{\alpha \eta}{\gamma k}))\W{X_{t,i}}{\txtki} + E_{t,i} \text{, where} \\ 
    E_{t,i} & \lesssim \frac{\eta L^{2}}{\alpha\gamma k}\e\|\tutki^+ - {\tilde U}_{t,i}\|^2 + \frac{\eta^3L^4}{\gamma^3 k}\e\|\tutki^+ - \tilde U_{t,0}\|^4 \\ & + \frac{5^p\eta^{p+2} k^{p-1} L^{2p+2}}{\gamma^{p+2}}\e\|\tutki^+ - \tilde U_{t,0}\|^{2p+2}.
\end{align*}
We now bound each of the error terms individually. Recall $N_t:= \sum_{i=0}^{k-1}H_{t,i}$ and let $M_{t,k}$ be as defined in Lemma~\ref{ULMCauxiliary}. 
\begin{align*}
    \frac{\eta L^{2}}{\alpha\gamma k}\e\|\tutki^+ - {\tilde U}_{t,i}\|^2 & \leq \frac{\eta L^{2}}{\alpha\gamma k}\e\Bigg[\sum_{j=0}^{i-1} k H_{t,j}\Bigg\|A_\mathcal M\Big(\frac{(i-1-j)\eta}{k}\Big)G_\mathcal M\Big(\ebk\Big)\begin{bmatrix}
        \nabla f(\utkt) - \nabla f(\tutki^+)\\0
    \end{bmatrix}\Bigg\|\Bigg]^2 \\
    & \lesssim \frac{\eta L^{2}}{\alpha\gamma k}\e \Big[\sum_{j=0}^{i-1} H_{t,j} \cdot {\eta^2L} \|{\tilde U}_{t,j}^+ - \tilde U_{t,0}\|\Big]^2 \\ &
    \lesssim \frac{\eta^5 L^4}{\alpha \gamma k} \e\Big[N_t^2\sup_{0 \leq j< k}\|{\tilde U}_{t,j}^+ - \tilde U_{t,0}\|^2\Big] \\& \lesssim \frac{\eta^5 L^4}{\alpha \gamma k} \e\Big[\sup_{0 \leq j< k}\|{\tilde U}_{t,j}^+ - \tilde U_{t,0}\|^2\Big]  \\ & 
    \lesssim \frac{\eta^7 L^4}{\alpha \gamma k}\e\|\vtkt\|^2 + \frac{\eta^9 L^4}{\alpha\gamma k}\e\|\nabla f(\utkt)\|^2 + \frac{\eta^5 L^4}{\alpha\gamma k}\e[\mtk^2] \\ & 
    \lesssim \frac{\eta^7 L^4}{\alpha \gamma k}\e\|\vtkt\|^2 + \frac{\eta^9 L^4}{\alpha\gamma k}\e\|\nabla f(\utkt)\|^2 + \frac{\eta^8 L^4}{\alpha k}(d+\log k).
\end{align*}
In the first inequality, we have used item 2 of Lemma \ref{ULMCauxiliary_2}. In the second, we have used item 1 of Lemma \ref{ULMCauxiliary_2} and Assumption \ref{well_conditioned_F}. In the fourth we have used that $N_t$ is independent of the iterates, and $\e[N_t]^2 \lesssim 1.$ In the fifth and last inequalities, we have used items 1 and 2 of Lemma \ref{ULMCauxiliary} respectively, with the assumption that $\gamma\eta$ is bounded. 
\begin{align*}
    \frac{\eta^3L^4}{\gamma^3 k}\e\|{\tilde U}_{t,i}^+ - & \tilde U_{t,0}\|^4  \lesssim \frac{\eta^7 L^4}{\gamma^3 k}\e\|\vtkt\|^4 + \frac{\eta^{11}L^4}{\gamma^3 k}\e\|\nabla f(\utkt)\|^4 + \frac{\eta^3 L^4}{\gamma^3 k}\e[\mtk^4] \\& \lesssim \frac{\eta^7 L^4}{\gamma^3 k}\e\|\vtkt\|^4 + \frac{\eta^{11}L^4}{\gamma^3 k}\e\|\nabla f(\utkt)\|^4 + \frac{\eta^9 L^4}{\gamma k}(d+\log k)^2.
\end{align*}
The above inequality follows from items 1 and 2 of Lemma \ref{ULMCauxiliary}, with the assumption that $\gamma\eta$ is bounded. Now, for some constant $\lambda_p$ depending only on $p$:
\begin{align*}
    \frac{5^p\eta^{p+2} k^{p-1} L^{2p+2}}{\gamma^{p+2}}&\e\|{\tilde U}_{t,i}^+ - \tilde U_{t,0}\|^{2p+2} \leq \lambda_p' \Big[\frac{\eta^{3p+4}k^{p-1}L^{2p+2}}{\gamma^{p+2}}\e\|\vtkt\|^{2p+2} \\ & +\frac{\eta^{5p+6}k^{p-1}L^{2p+2}}{\gamma^{p+2}}\e\|\nabla f(\utkt)\|^{2p+2} + \frac{\eta^{p+2}k^{p-1}L^{2p+2}}{\gamma^{p+2}}\e[\mtk^{2p+2}]\Big] \\ & \leq \lambda_p\Big[\frac{\eta^{3p+4}k^{p-1}L^{2p+2}}{\gamma^{p+2}}\e\|\vtkt\|^{2p+2} +\frac{\eta^{5p+6}k^{p-1}L^{2p+2}}{\gamma^{p+2}}\e\|\nabla f(\utkt)\|^{2p+2} \\ & +\frac{\eta^{4p+5}k^{p-1}L^{2p+2}}{\gamma}(d+\log k)^{p+1}\Big].
\end{align*}
As before, the above inequality follows from items 1 and 2 of Lemma \ref{ULMCauxiliary}.
\end{proof}

\section{Proof of Corollary~\ref{ULMCcomplexity}}
\label{sec:ULMCcomplexity_proof}
\begin{proof}
    Triangle inequality on $\mathcal W_2$ gives 
    \[\W{\tilde U_{N,0}}{\pi} \lesssim \W{\tilde U_{N,0}}{U_{N,0}}+\W{U_{N,0}}{\pi} \leq \W{\tilde X_{N,0}}{X_{N,0}}+\W{U_{N,0}}{\pi}.\] Under the conditions of the Corollary, we show that both these terms are $\lesssim \frac{\epsilon^2 d}{\alpha}.$ Recall the following Theorem for the convergence of Underdamped LMC. 
    \begin{theorem}[Corollary of Theorem 2, Dalalyan \& Riou-Durand~\cite{dalalyan2020sampling}]
        Let $f$ satisfy Assumption \ref{well_conditioned_F}. In addition, let the initial condition of ULMC be drawn from the product distribution $\mu = \mathcal N(0,\vI_d) \otimes \nu_0.$ For $\gamma = c\sqrt L$ and step-size $h = \frac{0.94\epsilon\sqrt{\alpha}}{ L\sqrt{2}},$ the distribution $\nu_k$ of the kth iterate of the ULMC algorithm satisfies $\W{\nu_k}{\pi} \leq \frac{\epsilon^2d}{\alpha}$ for $k \geq c_3\frac{\gamma}{\alpha h}\log \frac{\sqrt {2\alpha} \mathcal W_2(\nu_0,\pi)}{\epsilon \sqrt d}$ and some absolute constant $c_3.$        
    \end{theorem}
    With $k$ defined as in the Corollary we have $\ebk \lesssim \frac{\epsilon\sqrt{\alpha}}{L}.$ Note that the Theorem above is valid with an inequality $h \leq \frac{0.94\epsilon\sqrt{\alpha}}{ L\sqrt{2}}$ rather than equality, so we get $\W{U_{N,0}}{\pi} \leq \frac{\epsilon^2d}{\alpha}$ for $N \geq \frac{c_3 \gamma}{\alpha h}\log \frac{\sqrt {2\alpha} W_2(\nu_0,\pi)}{\epsilon \sqrt d}.$ It remains to be shown that $\W{\tilde X_{N,0}}{X_{N,0}} \lesssim \frac{\epsilon^2 d}{\alpha}.$ 
    Let $n$ be a natural number. Under our assumptions on $V_{0,0}$ and $U_{0,0}$, we have $\e||V_0||^{2n} = d^n$ and 
    \begin{align*}
        (f(\Psi_0)-f(\Psi_T))^{+} & \leq f(\Psi_0) - f(x^*) \\
            & \leq L||\Psi_0 - x^*||^2 \\& 
            \lesssim L||U_{0,0}-x^*||^2 + \frac{L}{\gamma^2}||V_{0,0}||^2
    \end{align*} 
    Under our assumptions, we thus get $\e|(f(\Psi_0)-f(\Psi_T))^{+}|^n \lesssim d^n.$ Moreover, we have $\log k \asymp \max(0, \log \frac{\eta L}{\epsilon\sqrt\alpha})\lesssim \log \frac{\kappa}{\epsilon}$ under the assumption that $\gamma\eta < c_0.$ 
    Now let $L_2 = c_2 \log \frac{\sqrt {2\alpha} \mathcal W_2(\nu_0,\pi)}{\epsilon \sqrt d},$ $L_3 = \log \frac{\kappa}{\epsilon}$ and apply Theorem \ref{ULMCmain} with $N$ as above, and 
    \begin{align*}
        \eta & \leq \min\Big(\tfrac{\epsilon^{1/3}}{\sqrt L},\tfrac{\epsilon^{1/3}}{\kappa^{1/6}L_2^{1/6}\sqrt L}, \tfrac{\epsilon^{1/3}\kappa^{1/6}}{d^{1/6}\sqrt L},\tfrac{\epsilon^{1/3}}{\kappa^{1/6}d^{1/6}L_2^{1/3}\sqrt L}, \tfrac{\epsilon^{\frac{p+2}{4p+3}}}{\kappa^{\tfrac{p/2-1}{4p+3}}d^{\tfrac{p}{4p+3}}\sqrt L} \\ & \tfrac{\epsilon^{\tfrac{p+2}{4p+3}}}{\kappa^{\tfrac{3p}{8p+6}}d^{\tfrac{p}{4p+3}}L_2^{\tfrac{p+1}{4p+3}}\sqrt L}, \tfrac{\epsilon^{1/3}d^{1/6}\kappa^{1/6}}{L_3^{1/6}\sqrt L},\tfrac{\epsilon^{1/3}d^{1/6}}{\kappa^{1/6}L_3^{1/3}\sqrt L},\tfrac{\epsilon^{\tfrac{p+2}{4p+3}}d^{\tfrac{1}{4p+3}}}{\kappa^{\tfrac{3p}{8p+6}}L_2^{\tfrac{p+1}{4p+3}}L_3^{\tfrac{p+1}{4p+3}}\sqrt L}\Big).
    \end{align*}
    Our assumption on $\epsilon$ is sufficient to ensure that the conditions of Theorem \ref{ULMCmain} are satisfied with $\eta$ as above. This gives $\W{\tilde X_{N,0}}{X_{N,0}}\leq \frac{\epsilon^2d}{\alpha},$ with $N = L_3\gamma(\alpha\eta)^{-1}$ as desired. 
\end{proof}

\section{Technical Results for ULMC}
\subsection{Proof of Proposition~\ref{ULMCamgm}}
\label{sec:ULMCamgmproof}

The following proposition provides useful bounds on the operator norms of $\Gamma_\mathcal M$ and $\gm$ based on Taylor series expansion. We refer to Section~\ref{sec:spectralNormBoundproof} for its proof.
\begin{proposition}\label{spectralNormBounds} Let $\|\cdot\|$ denote the operator norm of a matrix, and $\|\cdot\|_{\mathbb R^n}$ denote the Euclidean norm in dimension $n$. Let $p$ and $q$ denote arbitrary points in $\mathbb R^{d}$. Assume $\gamma h <c_0$ for some sufficiently small constant $c_0>0,$ and Assumption \ref{well_conditioned_F}. Then the following inequalities are true. 
    \begin{enumerate}
        \item $\|\Gamma_\mathcal M(h)\|^2 \lesssim \frac{h}{\gamma}.$
        \item $\|G_\mathcal M(h)^T \Gamma_\mathcal M(h)^{-2}G_\mathcal M(h)\| \leq \frac{h}{\gamma}.$ 
        \item $\Big\|G_\mathcal M(\eta) \begin{bmatrix}
            \nabla f(p) - \nabla f(q)\\0
        \end{bmatrix}\Big\|_{\mathbb R^{2d}} \lesssim \frac{hL}{\gamma}\|p-q\|_{\mathbb R^d}.$
    \end{enumerate}
\end{proposition}

We now prove Proposition~\ref{ULMCamgm}. 
\begin{proof}
    Let $h = \ebk$, and \[T\Big(\begin{bmatrix}
        u\\v
    \end{bmatrix}\Big) = A_\mathcal M(h)\begin{bmatrix}
        u\\v
    \end{bmatrix} + G_\mathcal M(h)\begin{bmatrix}
        -\nabla f(u)\\0
    \end{bmatrix}.\]
    Given Assumption \ref{well_conditioned_F}, with $\gamma = c\sqrt L$ for some $c \geq \sqrt 2,$ the map $T$ is Lipschitz with $\|T\|_\text{Lip} \leq 1 - \frac{\alpha}{\sqrt L}h + O(Lh^2)$ (Lemma 16, Zhang et al.~\cite{zhang2023improveddiscretizationanalysisunderdamped}.) Under our assumptions we have $L(\ebk)^2 \lesssim \frac{\alpha}{\sqrt L}\cdot \ebk,$ and $T$ is thus a contraction with parameter $1-\Omega(\frac{\alpha\eta}{\sqrt L k}).$  Under the event $\g$, we have 
    \begin{align*}
        \Big\|X_{t,i+1} - \tilde X_{t,i+1}\Big\|^2 & = \Big\|T(x) - T(y)\Big\|^2 + \Big\|\Gamma_\mathcal M\Big(\ebk\Big)(\y - \tilde Z_{t,i})\Big\|^2 \\ & + 2 \Big\langle \Gamma_\mathcal M\Big(\ebk\Big)(\y - \tilde Z_{t,i}), T(x) - T(y) \Big\rangle.
    \end{align*}
     By the definition of $\zt$, 
    \[\Gamma_\mathcal M\Big(\ebk\Big)(\zt - \y) = \Gamma_\mathcal M\Big(\ebk\Big)(\z+S_{t,i}-\y) + G_\mathcal M\Big(\ebk\Big)\begin{bmatrix}
        \nabla f(\tutki) - \nabla f(\tutki^+)\\0
    \end{bmatrix}.\]
    Conditioned on $\g$, $(\bern)$ is zero mean and $\z$ and $\y$ are standard Gaussians. This gives $\e[\Gm(\ebk)(\z' - \tilde Z_{t,i})|\g] = \gm\Big(\ebk\Big) \begin{bmatrix}
        \nabla f(\tutki) - \nabla f(\tutki^+)\\0
    \end{bmatrix}$. By item 3 of Proposition \ref{spectralNormBounds}, and the contractivity of $T$, we get 
    \begin{align*}
        \e[\|X_{t,i+1} - \tilde X_{t,i+1}\|^2 | \g] & \leq (1-\Omega(\frac{\alpha\eta }{\gamma k}))\|x-y\|^2 + \mathcal O\Big[\Big\|\Gm\Big(\ebk\Big)(\y - \tilde Z_{t,i})\Big\|^2  \\ & + \frac{\eta L}{k\gamma}\|u^+ - \tilde u\|\cdot \|x-y\|\Big].
    \end{align*}

    An application of the AM-GM inequality gives 
    \[\frac{\eta L}{\gamma k}\|u^+- \tilde u\|\cdot\|x-y\| \lesssim \frac{\eta L^2}{\tau \alpha \gamma k}\|u^+ - \tilde u\|^2 + \tau \frac{\alpha \eta }{\gamma k}\|x-y\|^2,\]
    Where $\tau>0$ is arbitrary. Choose $\tau$ small enough so that the second term can be absorbed into $(1-\Omega(\frac{\alpha\eta }{\gamma k}))\|x-y\|^2$. 
    We also have
    \begin{align*}
    \e[\|\Gm\Big(\ebk\Big)(\zt - \y)\|^2|\g] & \lesssim \|\gm\Big(\ebk\Big)\begin{bmatrix}
        \nabla f(\tutki) - \nabla f(\tutki^+)\\0
    \end{bmatrix}\|^2 \\ & + \|\Gm\Big(\ebk\Big)\|^2 \cdot \e[\|\z+S_{t,i}-\y\|^2|\g].
    \end{align*}
  
    By item 3 of Proposition \ref{spectralNormBounds}, $\|\gm\Big(\ebk\Big)\begin{bmatrix}
        \nabla f(\tutki) - \nabla f(\tutki^+)\\0
    \end{bmatrix}\|^2 \leq \frac{\eta^2 L^2}{k^2 \gamma^2}\|\tilde u-u^+\|^2 \leq \frac{\eta L^2}{\alpha \gamma k}\|\tilde u-u^+\|^2;$ and by item 1, $\|\Gm\Big(\ebk\Big)\|^2 \leq \frac{\eta}{\gamma k}.$
    \begin{align*}
    \e[\|\Gm(\ebk)(\zt - \y)\|^2|\g] & \lesssim \frac{\eta L^2}{\alpha \gamma k}\|\tilde u-u^+\|^2 + \frac{\eta}{\gamma k}\e[\|\z+S_{t,i}-\y\|^2|\g].
    \end{align*}
\end{proof}

\subsection{Proof of Proposition~\ref{spectralNormBounds}}
\label{sec:spectralNormBoundproof}
\begin{proof}
    The eigenvalues of $\Gamma_\mathcal M(h)^2$ are 
    \begin{align*}
         & E_1 = \frac{\exp(-2\gamma h)}{\gamma^2}(a-b),  & E_2 = \frac{\exp(-2\gamma h)}{\gamma^2}(a+b)
    \end{align*}
    where $a = -1+\exp(2\gamma h)(1+2\gamma h),$ and \[b = \sqrt{1-32\exp(3\gamma h)+2\exp(2\gamma h)(7+2\gamma h)+ \exp(4\gamma h)(17-4\gamma h+4\gamma^2 h^2)}.\]
    
    Taylor expansion in the variable $h$ gives 
    \[\frac{\exp(-2\gamma h)}{\gamma^2}a = \frac{4h}{\gamma}-2h^2+ \frac{4\gamma h^3}{3} + O(\gamma^2 h^4),\]
    \[\frac{\exp(-2\gamma h)}{\gamma^2}b = \frac{4h}{\gamma}-2h^2+ \frac{7\gamma h^3}{6} + O(\gamma^2 h^4).\]
    
    As a result, the eigenvalues $E_1$ and $E_2$ are of order $\gamma h^3$ and $\frac{h}{\gamma}$ respectively, with $E_2 \geq E_1$ being the spectral norm of $\Gamma_\mathcal M(h)^2.$
    We compute the inverse:
    \[\det(\Gamma_\mathcal M(h)^2) = \frac{8}{\gamma^4}\exp(-2\gamma h)(-1+\exp(\gamma h))(2+\gamma h + \exp(\gamma h)(-2+\gamma h)),\] 
    
    \[\Gamma_\mathcal M(h)^{-2} = \det(\Gamma_\mathcal M(h)^2)^{-1}\begin{bmatrix}
    \frac{4(1-\exp(-\gamma h) - (1-\exp(2\gamma h)) + 2\gamma h}{\gamma^2}\vI_d & -\frac{2\gamma h - (1-\exp(2\gamma h))}{\gamma^2}\vI_d \\
    -\frac{2\gamma h - (1-\exp(2\gamma h))}{\gamma^2}\vI_d & \frac{4(1-\exp(-\gamma h) + (1-\exp(2\gamma h)) + 2\gamma h}{\gamma^2}\vI_d
    \end{bmatrix}.\]
    
    An explicit computation gives \[G_\mathcal M(h)^T \Gamma_\mathcal M(h)^{-2}G_\mathcal M(h) = \begin{bmatrix}
        \frac{h}{2\gamma} & 0 \\ 0 & 0
    \end{bmatrix}.\] 
    A Taylor expansion on the entries of $G_\mathcal M(h)$ shows 
    \[G_\mathcal M(h) = \begin{bmatrix} \frac{h^2}{2} + O(\gamma h^3) & 0 \\
        \frac{2h}{\gamma}-\frac{h^2}{2} + O(\gamma h^3) & 0 \\
    \end{bmatrix}.\] Item 3 of the proposition follows from this and the smoothness of $f$ -- Assumption \ref{well_conditioned_F}. 
\end{proof}
\begin{lemma}\label{smooth_gradient_bounds}
    Assume $\pi = \exp(-f)$ is $L$-smooth, and let $p \in \sN$. Then 
    \[\int_{\sR^d}||\nabla f(x)||^{2p} d\pi(x) \leq \prod_{n=1}^p(2n-1) \cdot (Ld)^p.\]
\end{lemma}
\begin{proof}
This is a generalization of \cite[Lemma 11]{vempala2022rapidconvergenceunadjustedlangevin}. By definition, we have
\begin{align*}
    \int_{\sR^d}||\nabla f(x)||^{2p} d\pi(x) & = \int_{\sR^d} \exp(-f(x))\Big[\sum_{i=1}^d \Big(\frac{\partial f}{\partial x_i}\Big)^2\Big]^p dx
    \intertext{By Jensen's inequality, we get}
    & \leq d^{p-1} \sum_{i=1}^d\int_{\sR^d} \exp(-f(x)) \Big(\frac{\partial f}{\partial x_i}\Big)^{2p}dx \\
    \intertext{Applying integration by parts along $x_i,$ we get}
    & = d^{p-1} (2p-1)\sum_{i=1}^d\int_{\sR^d} \exp(-f(x)) \Big(\frac{\partial^2 f}{\partial x_i^2}\Big)\Big(\frac{\partial f}{\partial x_i}\Big)^{2p-2}dx \\
    \intertext{Since $f$ is $L$-smooth, we have $\frac{\partial^2f}{\partial x_i^2} \leq L.$}
    & \leq Ld^{p-1}(2p-1)\sum_{i=1}^d\int_{\sR^d}\exp(-f(x))\Big(\frac{\partial f}{\partial x_i}\Big)^{2p-2}dx \\
    \intertext{By a repeated application of integration by parts, we get}
    & \leq L^p d^{p-1} \prod_{n=1}^p(2n-1) \sum_{i=1}^d \int_{\sR^d} \exp(-f)dx \\
    \intertext{Since $\pi$ is a probability measure, $\int\exp(-f)=1.$}
    & = \prod_{n=1}^p(2n-1)(Ld)^p.
\end{align*}
\end{proof}

\newpage
\end{document}